\documentclass[11pt,english]{extarticle}
\usepackage[latin9]{inputenc}
\usepackage[letterpaper]{geometry}
\usepackage{mathrsfs}
\usepackage{amsmath}
\usepackage{amsthm}
\usepackage{amssymb}
\usepackage{graphicx}
\usepackage{esint}

\makeatletter
  \theoremstyle{remark}
  \newtheorem{rem}{\protect\remarkname}
\theoremstyle{plain}
\newtheorem{thm}{\protect\theoremname}
  \theoremstyle{plain}
  \newtheorem{lem}{\protect\lemmaname}

\usepackage{dsfont}
\usepackage{relsize}
\usepackage{subdepth}
\allowdisplaybreaks

\usepackage{url}

\DeclareFontFamily{OT1}{pzc}{}
\DeclareFontShape{OT1}{pzc}{m}{it}{<-> s * [1.200] pzcmi7t}{}
\DeclareMathAlphabet{\mathpzc}{OT1}{pzc}{m}{it}

\usepackage[noadjust]{cite}
\usepackage{filecontents}

\usepackage{stfloats}

\usepackage{babel}
\renewcommand\footnotemark{}

\@ifundefined{showcaptionsetup}{}{%
 \PassOptionsToPackage{caption=false}{subfig}}
\usepackage{subfig}
\makeatother

\usepackage{babel}
  \providecommand{\lemmaname}{Lemma}
  \providecommand{\remarkname}{Remark}
\providecommand{\theoremname}{Theorem}

\begin{document}

\title{\textbf{Uniform $\varepsilon$-Stability of Distributed Nonlinear}\\
\textbf{Filtering over DNAs: Gaussian-Finite HMMs}}

\author{Dionysios S. Kalogerias\thanks{The Authors are with the Department of Electrical \& Computer Engineering,
Rutgers, The State University of New Jersey, 94 Brett Rd, Piscataway,
NJ 08854, USA. e-mail: \{d.kalogerias, athinap\}@rutgers.edu.}\thanks{Part of this work \cite{KalPet_Asilomar_DFilt_2015} was presented
at the \textit{49th Asilomar Conference on Signals, Systems \& Computers
(Asilomar 2015)}. This work is supported by the National Science Foundation
(NSF) under Grants CCF-1526908 \& CNS-1239188.} and Athina P. Petropulu}

\date{September 2016}

\maketitle
\textbf{\vspace{-30pt}
}
\begin{abstract}
In this work, we study stability of distributed filtering of Markov
chains with finite state space, partially observed in conditionally
Gaussian noise. We consider a nonlinear filtering scheme over a Distributed
Network of Agents (DNA), which relies on the distributed evaluation
of the likelihood part of the centralized nonlinear filter and is
based on a particular specialization of the Alternating Direction
Method of Multipliers (ADMM) for fast average consensus. Assuming
the same number of con\emph{s}ensus steps between any two consecutive
noisy measurements for each sensor in the network, we fully characterize
a minimal number of such steps, such that the distributed filter remains
uniformly stable with a prescribed accuracy level, $\varepsilon\in\left(0,1\right]$,
within a finite operational horizon, $T$, and across all sensors.
Stability is in the sense of the $\ell_{1}$-norm between the centralized
and distributed versions of the posterior at each sensor, and at each
time within $T$. Roughly speaking, our main result shows that uniform
$\varepsilon$-stability of the distributed filtering process depends
only loglinearly on $T$ and (roughly) the size of the network, and
only logarithmically on $1/\varepsilon$. If this total loglinear
bound is fulfilled, any additional consensus iterations will incur
a fully quantified further exponential decay in the consensus error.
Our bounds are universal, in the sense that they are independent of
the particular structure of the Gaussian Hidden Markov Model (HMM)
under consideration.
\end{abstract}
\textbf{\textit{$\quad$}}\textbf{Keywords.} Nonlinear Filtering,
Markov Chains, ADMM, Average Consensus, Distributed State Estimation,
Hidden Markov Models.

\setlength{\textfloatsep}{5pt} 

\section{Introduction}

Distributed state estimation and tracking of partially observed processes
constitute central problems in modern Distributed Networks of Agents
(DNAs), where, in the absence of a powerful fusion center, a group
of power constrained sensing devices with limited computation and/or
communication capabilities receive noisy measurements of some common,
possibly rapidly varying process of global network interest. As also
surveyed in \cite{DPF_SPM_2013}, important applications of DNAs,
where potentially the aforementioned problems arise, include environmental
and agricultural monitoring \cite{APP1_2010}, health care monitoring
\cite{APP2_2010}, pollution source localization \cite{APP4_2007},
surveillance \cite{APP5_2009}, chemical plume tracking \cite{APP3_2004_zhao2004wireless},
target tracking \cite{APP3_2004_zhao2004wireless} and habitat monitoring
\cite{APP6_2002}, to name a few.

In the context of linear state estimation, in recent years and in
parallel with the rapid development of fast averaging consensus protocols
and algorithms in networked measurement systems \cite{Consensus1_2004_Olfati-Saber,Consensus2_2005_Olfati-Saber,Boyd_2006Gossip,Consensus3_2010,Consensus4_2010_Oreshkin,Consensus_Emiliano_2011},
there has been extensive research on the important basic problem of
distributed Kalman filtering over DNAs, under several different perspectives,
such as utilizing control theoretic consensus algorithms \cite{DKF1_2005_Olfati-Saber,DKF4_Olfati-Saber,DKF2_2007_Olfati-Saber},
the sign-of-innovations approach \cite{DKF5_2006_SOIKF_Riberio},
the Alternating Direction Method of Multipliers (ADMM) \cite{DKF6_2008_Schizas},
which will also be considered in this work and others, based on more
customized consensus strategies \cite{DKF3_2006_alriksson,DKF7_2008}.
For a relatively complete list of references, see the extensive survey
\cite{DKF_2013_Survey}.

Although not as rich, the literature concerning the problem of distributed
nonlinear filtering of general Markov processes is itself quite extensive.
Since, in general, most nonlinear filters do not admit finite dimensional
representations, the focus in this case is the derivation of distributed
schemes for the implementation of well defined, finite dimensional
nonlinear filtering approximations, which would allow for efficient,
real time state estimation. In this direction, successful examples
include distributed particle filters \cite{DPF_SPM_2013}, distributed
extended Kalman filters \cite{DKF4_Olfati-Saber,DKF5_2006_SOIKF_Riberio},
as extensions to the linear case mentioned above and, more recently,
the Bayesian consensus filtering approach proposed in \cite{BCF1_2014}.

In this paper, we focus on distributed state estimation of Markov
chains with \textit{finite state space}, partially observed in \textit{conditionally
Gaussian noise}. Hereafter, we will use the term Gaussian Hidden Markov
Model (HMM). This special case is, however, of broad practical interest.
Further, it is known that the posterior probability measure of the
chain at a certain time, relative to the measurements obtained so
far, admits a recursive representation \cite{Elliott1994Exact,Elliott1994Hidden}.
This posterior can subsequently be used in order to produce any optimal
estimate of interest, such as the MMSE estimator of the chain, the
respective MAP estimator, etc. The noisy observations of the chain
are obtained distributively over a DNA, whose connectivity pattern
follows a connected Random Geometric Graph (RGG) \cite{RGG_Dall2002,RGG_Penrose2003}.

Under this setting, we consider a distributed filtering scheme, which
relies on the distributed evaluation of the likelihood part of the
centralized posterior under consideration and is based on a particular
specialization of the ADMM for fast average consensus \cite{Consensus_Emiliano_2011}.
Our idea is similar to the one conveyed by the respective formulation
for particle filters \cite{DPF_SPM_2013}. In order to account for
rapidly varying Markov chains with arbitrary statistical structure,
in our formulation, nonlinear filtering and distributed average consensus
are implemented in \textit{different time scales}, with the message
exchange rate between any sensor and its neighbors being much larger
than the rate of measurement acquisition \cite{DKF4_Olfati-Saber}.
This is a realistic assumption in systems where measurements are acquired
at relatively distant times; for example, every minute, hour or day,
or even more dynamically, possibly at event triggered time instants
or a sequence of stopping times. Additionally, in this paper we assume
perfect communications among sensors. This is a reasonable assumption,
provided, for instance, that the filtering process is implemented
in a higher than the physical layer of the network under consideration.
Our contributions are summarized as follows:

1) First, based on the matrix equivalent formulation of the ADMM for
implementing average consensus and the respective convergence results
presented in \cite{Consensus_Emiliano_2011}, we focus on optimizing
the respective consensus error bound, with respect to two free parameters:
a scalar $\epsilon>0$ and the second largest eigenvalue of a symmetric
(doubly) stochastic matrix ${\bf S}$ (\cite{Consensus_Emiliano_2011},
also see Section \ref{sec:SYSTEM_MODEL}). Under indeed mild conditions,
we show analytically that the error bound is optimized (in a certain
sense) at an explicitly defined $\epsilon$ and by choosing ${\bf S}$
such that its second largest eigenvalue is minimized, the latter being
a convex optimization problem, which can be efficiently solved \cite{Boyd_2006Gossip}.
Our results essentially complement earlier work on ADMM-based consensus,
previously presented in \cite{Consensus_Emiliano_2011}.

2) Then, utilizing the aforementioned optimized bounds and assuming
the same number of consensus iterations between any two consecutive
measurements for each sensor in the network, we fully characterize
a minimal number of such iterations required, such that the distributed
filter remains \textit{uniformly} stable with a prescribed accuracy
level, $\varepsilon\in\left(0,1\right]$, within a finite operational
horizon, $T$. Uniform stability is in the sense of the supremum of
the $\ell_{1}$-norm between the centralized and distributed versions
of the posterior in each sensor, over all sensors and over all times
within the operational horizon of interest. Roughly speaking, our
result shows that, under very reasonable assumptions, the stability
of the distributed filtering process depends only \textit{loglinearly}
on $T$ and the total number of measurements in the network, $N$,
and only \textit{logarithmically} on $1/\varepsilon$. Additionally,
if this total loglinear bound is fulfilled, any additional consensus
iterations will incur an \textit{exponential decrease} in the aforementioned
$\ell_{1}$-norm. The result is \textit{fundamental }and\textit{ universal},
since, apart from the assumed conditional Gaussianity of the observations,
it is virtually \textit{independent of the internal structure} of
the particular HMM under consideration. This fact makes it particularly
attractive in highly heterogeneous DNAs.

The problem of distributed inference in HMMs has been considered earlier
in \cite{DHMM1_2010}. However, the approach taken in \cite{DHMM1_2010}
is distinctly different from our proposed approach. In particular,
the formulation in \cite{DHMM1_2010} is based on control theoretic
consensus \cite{Consensus2_2005_Olfati-Saber} and stochastic approximation,
while our distributed filtering formulation is based on optimized
ADMM-based consensus. Also, the method of \cite{DHMM1_2010} requires
certain assumptions on the statistical structure of the hidden Markov
chain (e.g., primitiveness), while our work makes no such assumptions.
Further, although the analysis presented in \cite{DHMM1_2010} does
provide some limited convergence guarantees, it does not provide any
results on the rate of convergence of the obtained distributed estimators.
Here, we provide a fully tractable stability analysis of the distributed
filtering scheme considered, with explicit, optimistic and universal
bounds on the rate of convergence, as well as the degree of consensus
achieved. Another important difference is that, in \cite{DHMM1_2010},
filtering and consensus are implemented simultaneously. Such setting
excludes cases where the underlying Markov chain is very rapidly varying
and cannot fully exploit the benefits of heterogeneity in the information
observed by the sensors of the DNA under consideration. This is due
to the fact that, in \cite{DHMM1_2010}, global consensus cannot,
in general, be guaranteed within a reasonable and quantitatively predictable
error margin. Contrary to \cite{DHMM1_2010}, the distributed filtering
scheme advocated herein efficiently exploits sensor heterogeneity,
since uniform diffusion of local information is achieved across the
network.

The paper is organized as follows. In Section II, we present the system
model in detail, along with some very basic preliminaries on nonlinear
filtering. Section III introduces our distributed filtering formulation
and, subsequently, focuses on the optimization of the consensus error
bounds produced by the ADMM iterations. In Section IV, we study stability
of the distributed filtering scheme considered in the sense mentioned
above, and we present the relevant results, along with complete proofs.
In Section V, we present some numerical simulations, experimentally
validating some of the properties of the proposed approach, as well
as a relevant discussion. Finally, Section VI concludes the paper.

\textit{Notation}: In the following, the state vector will be represented
as $X_{t}$, and all other matrices and vectors will be denoted by
boldface uppercase and boldface lowercase letters, respectively. Real
valued random variables will be denoted by uppercase letters. Calligraphic
letters and formal script letters will denote sets and $\sigma$-algebras,
respectively. The operators $\left(\cdot\right)^{\boldsymbol{T}}$,
$\lambda_{min}\left(\cdot\right)$ and $\lambda_{max}\left(\cdot\right)$
will denote transposition, minimum and maximum eigenvalue, respectively.
The $\ell_{p}$-norm of a vector $\boldsymbol{x}\in\mathbb{R}^{n}$
is $\left\Vert \boldsymbol{x}\right\Vert _{p}\triangleq\left(\sum_{i=1}^{n}\left|x\left(i\right)\right|^{p}\right)^{1/p}$,
for all naturals $p\ge1$. For any Euclidean space, ${\bf I}$ will
denote the respective identity operator of appropriate dimension.
For any set ${\cal A}$, its complement is denoted as ${\cal A}^{c}$.
Additionally, we employ the identifications $\mathbb{N}^{+}\equiv\left\{ 1,2,\ldots\right\} $,
$\mathbb{N}_{n}^{+}\equiv\left\{ 1,2,\ldots,n\right\} $, $\mathbb{N}_{n}\equiv\left\{ 0\right\} \cup\mathbb{N}_{n}^{+}$
and $\mathbb{N}_{n}^{m}\equiv\mathbb{N}_{n}^{+}\setminus\mathbb{N}_{m-1}^{+}$,
for any positive naturals $n>m$.

\section{\label{sec:SYSTEM_MODEL}System Model \& Preliminaries}

In this section, we first present a generic model of the class of
systems under consideration, where multiple, possibly indirectly connected
sensors observe noisy versions of a common, hidden (i.e., unobserved),
stochastically evolving underlying signal of interest. Second, we
briefly discuss the centralized solution to problem of inference of
the underlying signal on the basis of the observations at the sensors.

\subsection{System Model: HMMs over DNAs}

Without loss of generality, we consider a wireless network consisting
of $S$ sensors, located inside the fixed square geometric region
${\cal S}\triangleq\left[0,1\right]^{2}$. The connectivity pattern
of the DNA under consideration is assumed to obey an \textit{undirected}
RGG model \cite{RGG_Dall2002,RGG_Penrose2003}. According to such
a model, the positions of the sensors, ${\bf p}_{i}\in{\cal S},i\in\mathbb{N}_{S}^{+}$,
are chosen uniformly at random in ${\cal S}$, whereas two sensors
$i$ and $j$ are considered connected if and only if $\left\Vert {\bf p}_{i}-{\bf p}_{j}\right\Vert _{2}\le r$,
where $r\in\left(0,\sqrt{2}\right]$ denotes the connectivity threshold
of the network. Throughout the paper, we assume that the underlying
RGG of the network is \textit{simply connected}, which constitutes
an event happening with very high probability for sufficiently large
number of sensors and/or sufficiently large connectivity threshold
\cite{RGG_Penrose2003}. Also, the one hop neighborhood of each sensor
$i$, including itself, is denoted as $\mathsf{N}_{i}$, for all $i\in\mathbb{N}_{S}^{+}$. 

At each discrete time instant $t\in\mathbb{N}$, sensor $i$ measures
the vector process ${\bf y}_{t}^{i}\in\mathbb{R}^{N_{i}\times1}$,
for all $i\in\mathbb{N}_{S}^{+}$, where $\sum_{i\in\mathbb{N}_{S}^{+}}N_{i}\triangleq N$.
It is assumed that the stacked process ${\bf y}_{t}\triangleq\left[\hspace{-2pt}\left({\bf y}_{t}^{1}\right)^{\boldsymbol{T}}\hspace{-2pt}\hspace{-2pt}\,\ldots\,\hspace{-2pt}\left({\bf y}_{t}^{S}\right)^{\boldsymbol{T}}\right]^{\boldsymbol{T}}\hspace{-2pt}\hspace{-2pt}\hspace{-2pt}\in\mathbb{R}^{N\times1}$
constitutes a noisy version (i.e., a functional) of a (for simplicity)
time-homogeneous Markov chain $X_{t}$, called the \textit{state},
with finite state space ${\cal X}$ of cardinality $L\in\mathbb{N}^{+}$.
Conditioned on $X_{t}$, the process ${\bf y}_{t}$ is distributed
according to a jointly Gaussian law as
\begin{equation}
\left.{\bf y}_{t}\right|X_{t}\sim{\cal N}\left(\boldsymbol{\mu}_{t}\left(X_{t}\right),\boldsymbol{\Sigma}_{t}\left(X_{t}\right)\right),\label{eq:Observation_Equation}
\end{equation}
where $\boldsymbol{\mu}_{t}:{\cal X}\mapsto\mathbb{R}^{N\times1}$
and $\boldsymbol{\Sigma}_{t}:{\cal X}\mapsto\mathbb{R}^{N\times N}$
constitute known measurable functionals of $X_{t}$, for all $t\in\mathbb{N}$.
In particular, $\boldsymbol{\Sigma}_{t}\left(\cdot\right)$ is assumed
to be of block diagonal form, constituting of the submatrices $\mathbb{R}^{N_{i}\times N_{i}}\ni\boldsymbol{\Sigma}_{t}^{i}\left(\boldsymbol{x}\right)\succ{\bf 0}$,
for all $i\in\mathbb{N}_{S}^{+}$, also implying that $\boldsymbol{\Sigma}_{t}\left(\boldsymbol{x}\right)\succ{\bf 0}$,
for all $\boldsymbol{x}\in{\cal X}$. As a result, the measurements
at each sensor are statistically independent, given $X_{t}$. Likewise,
we define $\boldsymbol{\mu}_{t}\left(\boldsymbol{x}\right)\hspace{-2pt}\triangleq\hspace{-2pt}\left[\hspace{-2pt}\left(\boldsymbol{\mu}_{t}^{1}\left(\boldsymbol{x}\right)\hspace{-2pt}\right)^{\boldsymbol{T}}\hspace{-2pt}\hspace{-2pt}\,\ldots\,\hspace{-2pt}\left(\boldsymbol{\mu}_{t}^{S}\left(\boldsymbol{x}\right)\hspace{-2pt}\right)^{\boldsymbol{T}}\right]^{\boldsymbol{T}}$,
where $\boldsymbol{\mu}_{t}^{i}\left(\boldsymbol{x}\right)\in\mathbb{R}^{N_{i}\times1}$,
for all $i\in\mathbb{N}_{S}^{+}$ and $\boldsymbol{x}\in{\cal X}$.
Further, the following additional technical assumptions are made.\textbf{\vspace{0.2cm}
}

\noindent \textbf{Assumption:} \label{Elementary_Definition_1-1}\textbf{(Boundedness)}
The quantities $\lambda_{max}\left(\boldsymbol{\Sigma}_{t}\left(\boldsymbol{x}\right)\right)$,
$\left\Vert \boldsymbol{\mu}_{t}\left(\boldsymbol{x}\right)\right\Vert _{2}$
are both uniformly upper bounded in $t\in\mathbb{N}$ and $\boldsymbol{x}\in{\cal X}$,
with finite bounds $\lambda_{sup}$ and $\mu_{sup}$, respectively.
For technical reasons, it is also true that $\lambda_{inf}\triangleq\inf_{t\in\mathbb{N}}\inf_{\boldsymbol{x}\in{\cal X}}\lambda_{min}\left(\boldsymbol{\Sigma}_{t}\left(\boldsymbol{x}\right)\right)\ge e>1$,
a requirement which can always be satisfied by normalization of the
observations.\textbf{\vspace{0.2cm}
}

Regarding the hidden Markov chain under consideration, it is true
that ${\cal X}\triangleq\left\{ \boldsymbol{x}_{1},\ldots,\boldsymbol{x}_{L}\right\} $,
where $\boldsymbol{x}_{i}$ denotes the $i$-th state of the chain,
for $i\in\mathbb{N}_{L}^{+}$. Although it is actually irrelevant,
for simplicity we will assume that $\boldsymbol{x}_{i}\in\mathbb{R}$,
for all $i\in\mathbb{N}_{L}^{+}$. The temporal dynamics of $X_{t}$
are expressed in terms of the initial distribution of the chain, encoded
in the vector $\pi_{-1}\in\mathbb{R}^{L\times1}$, as well as the
\textit{column stochastic} transition matrix $\boldsymbol{P}\in\left[0,1\right]^{L\times L}$,
defined, as usual, as $\boldsymbol{P}\left(i,j\right)\triangleq{\cal P}\left(\left.X_{t}\equiv\boldsymbol{x}_{i}\right|X_{t-1}\equiv\boldsymbol{x}_{j}\right),$
for all $\left(i,j\right)\in\mathbb{N}_{L}^{+}\times\mathbb{N}_{L}^{+}$.

\subsection{Centralized Recursive Estimates}

Let $\left\{ \mathscr{Y}_{t}\right\} _{t\in\mathbb{N}}$ be the complete
filtration generated by the observations ${\bf y}_{t}$. A quantity
of central importance in nonlinear filtering is the posterior probability
measure of the state $X_{t}$ given $\mathscr{Y}_{t}$, which, since
the state space of $X_{t}$ is finite, can be represented by a random
vector $\pi_{t\left|\mathscr{Y}_{t}\right.}\in\mathbb{R}^{L\times1}$.
Under the system model defined above, this random vector can be exactly
evaluated in real time as \cite{Elliott1994Hidden,Elliott1994Exact}
\begin{equation}
\pi_{t\left|\mathscr{Y}_{t}\right.}\equiv\dfrac{\boldsymbol{E}_{t}}{\left\Vert \boldsymbol{E}_{t}\right\Vert _{1}},\quad\forall t\in\mathbb{N},\label{eq:first}
\end{equation}
where the process $\boldsymbol{E}_{t}\in\mathbb{R}^{L\times1}$ satisfies
the linear recursion $\boldsymbol{E}_{t}\equiv\boldsymbol{\Lambda}_{t}\boldsymbol{P}\boldsymbol{E}_{t-1},$
for all $t\in\mathbb{N}$, with $\boldsymbol{E}_{-1}\triangleq\pi_{-1}$,
\begin{align}
\boldsymbol{\Lambda}_{t} & \hspace{-2pt}\triangleq\hspace{-2pt}\mathrm{diag}\left(\mathsf{L}_{t}\left(\boldsymbol{x}_{1}\right)\,\ldots\,\mathsf{L}_{t}\left(\boldsymbol{x}_{L}\right)\right)\in\mathbb{R}^{L\times L},\label{eq:LIKELIHOODS_1}\\
\mathsf{L}_{t}\left(\boldsymbol{x}\right) & \hspace{-2pt}\triangleq\hspace{-2pt}\dfrac{\exp\hspace{-2pt}\left(\hspace{-2pt}-\dfrac{1}{2}{\bf \overline{y}}_{t}^{\boldsymbol{T}}\left(\boldsymbol{x}\right)\boldsymbol{\Sigma}_{t}^{-1}\left(\boldsymbol{x}\right){\bf \overline{y}}_{t}\left(\boldsymbol{x}\right)\hspace{-2pt}\right)}{\sqrt{\det\left(\boldsymbol{\Sigma}_{t}\left(X_{t}\right)\right)}}\in\mathbb{R},\text{ and}\label{eq:LIKELIHOODS_2}\\
{\bf \overline{y}}_{t}\left(\boldsymbol{x}\right) & \hspace{-2pt}\triangleq\hspace{-2pt}{\bf y}_{t}-\boldsymbol{\mu}_{t}\left(\boldsymbol{x}\right)\in\mathbb{R}^{N\times1},\;\forall\boldsymbol{x}\in{\cal X}\text{ and }\forall t\in\mathbb{N}.
\end{align}
Of course, under these circumstances and defining ${\bf X}\triangleq\left[\boldsymbol{x}_{1}\,\ldots\,\boldsymbol{x}_{L}\right]\in\mathbb{R}^{1\times L}$,
the MMSE estimator of $X_{t}$ given $\mathscr{Y}_{t}$ can be easily
obtained as $\widehat{X}_{t}\triangleq\mathbb{E}\left\{ X_{t}\left|\mathscr{Y}_{t}\right.\right\} \equiv{\bf X}\pi_{t\left|\mathscr{Y}_{t}\right.},$
for all $t\in\mathbb{N}$. Similar estimates may be obtained for other
quantities of possible interest, such as the posterior error covariance
matrix between $X_{t}$ and $\widehat{X}_{t}$, etc. In the following,
we focus on the distributed estimation of the random probability measure
encoded in $\pi_{t\left|\mathscr{Y}_{t}\right.}$ (and not specific
functionals of it).
\begin{rem}
The structure of the nonlinear filtering procedure outlined above
may be interpreted as follows. Quantity \eqref{eq:LIKELIHOODS_2}
constitutes the \textit{conditional likelihood} of the observations,
\textit{given} that the state (the hidden process) is fixed. As a
result, the filtering recursion for $\boldsymbol{E}_{t}$ may be decomposed
into two intuitively informative parts, namely, a \textit{one-step
ahead prediction} step $\left(\boldsymbol{P}\boldsymbol{E}_{t-1}\right)$,
and a \textit{correction/update} step $\boldsymbol{\Lambda}_{t}\left(\boldsymbol{P}\boldsymbol{E}_{t-1}\right)$.
Then, (re)normalization by $\left\Vert \boldsymbol{E}_{t}\right\Vert _{1}$
converts the produced estimate into a valid probability measure.\hfill{}\ensuremath{\blacksquare}
\end{rem}

\section{\label{sec:Preliminaries}Distributed Estimation via the ADMM}

Employing somewhat standard notation, used, for instance, in \cite{Consensus_Emiliano_2011},
exploitting the block-diagonal structure of $\boldsymbol{\Sigma}_{t}$
and, therefore, of $\boldsymbol{\Sigma}_{t}^{-1}$ as well, and with
${\bf \overline{y}}_{t}^{i}\left(\boldsymbol{x}\right)\triangleq{\bf y}_{t}^{i}-\boldsymbol{\mu}_{t}^{i}\left(\boldsymbol{x}\right)$,
we may reexpress \eqref{eq:LIKELIHOODS_2} as
\begin{flalign}
\mathsf{L}_{t}\left(\boldsymbol{x}\right) & \equiv\exp\hspace{-2pt}\left(\hspace{-2pt}-\dfrac{1}{2}\hspace{-2pt}\left({\bf \overline{y}}_{t}^{\boldsymbol{T}}\left(\boldsymbol{x}\right)\hspace{-2pt}\boldsymbol{\Sigma}_{t}^{-1}\left(\boldsymbol{x}\right){\bf \overline{y}}_{t}\left(\boldsymbol{x}\right)\hspace{-2pt}+\hspace{-2pt}\log\det\hspace{-2pt}\left(\boldsymbol{\Sigma}_{t}\hspace{-2pt}\left(\boldsymbol{x}\right)\right)\right)\hspace{-2pt}\right)\nonumber \\
 & \equiv\exp\hspace{-2pt}\left(\hspace{-2pt}-\dfrac{1}{2}\sum_{i\in\mathbb{N}_{S}^{+}}\left[\left({\bf \overline{y}}_{t}^{i}\left(\boldsymbol{x}\right)\right)^{\boldsymbol{T}}\left(\boldsymbol{\Sigma}_{t}^{i}\left(\boldsymbol{x}\right)\right)^{-1}{\bf \overline{y}}_{t}^{i}\left(\boldsymbol{x}\right)+\log\det\hspace{-2pt}\left(\boldsymbol{\Sigma}_{t}^{i}\hspace{-2pt}\left(\boldsymbol{x}\right)\right)\right]\hspace{-2pt}\right)\hspace{-2pt},
\end{flalign}
for all $\boldsymbol{x}\in{\cal X}$ and $t\in\mathbb{N}$. Let
\begin{align}
\theta_{t}^{i}\left(\boldsymbol{x}\right) & \triangleq S\left({\bf \overline{y}}_{t}^{i}\left(\boldsymbol{x}\right)\right)^{\boldsymbol{T}}\left(\boldsymbol{\Sigma}_{t}^{i}\left(\boldsymbol{x}\right)\right)^{-1}{\bf \overline{y}}_{t}^{i}\left(\boldsymbol{x}\right)+S\log\det\left(\boldsymbol{\Sigma}_{t}^{i}\left(\boldsymbol{x}\right)\right),\;\forall i\in\mathbb{N}_{S}^{+},\;\forall\boldsymbol{x}\in{\cal X}\text{ and}
\end{align}
$\boldsymbol{\theta}_{t}\hspace{-2pt}\left(\boldsymbol{x}\right)\hspace{-2pt}\triangleq\hspace{-2pt}\hspace{-2pt}\left[\theta_{t}^{1}\hspace{-2pt}\left(\boldsymbol{x}\right)\hspace{-2pt}\,\ldots\hspace{-2pt}\,\theta_{t}^{S}\hspace{-2pt}\left(\boldsymbol{x}\right)\right]^{\boldsymbol{T}}\hspace{-2pt}\hspace{-2pt}\hspace{-2pt}\in\hspace{-2pt}\mathbb{R}^{S\times1}\hspace{-2pt},\forall t\hspace{-2pt}\in\hspace{-2pt}\mathbb{N}$.
Then, it is true that
\begin{align}
\theta_{t}\left(\boldsymbol{x}\right) & \triangleq\sum_{i\in\mathbb{N}_{S}^{+}}\dfrac{\theta_{t}^{i}\left(\boldsymbol{x}\right)}{S}\equiv{\bf \overline{y}}_{t}^{\boldsymbol{T}}\left(\boldsymbol{x}\right)\hspace{-2pt}\boldsymbol{\Sigma}_{t}^{-1}\hspace{-2pt}\left(\boldsymbol{x}\right){\bf \overline{y}}_{t}\hspace{-2pt}\left(\boldsymbol{x}\right)+\log\det\hspace{-2pt}\left(\boldsymbol{\Sigma}_{t}\hspace{-2pt}\left(\boldsymbol{x}\right)\right)\hspace{-2pt}.\label{eq:CONSENSUS_1}
\end{align}
Of course, sensor $i$ of the network observes $\theta_{t}^{i}$,
for all $t\in\mathbb{N}$. It is then clear that knowledge of $\theta_{t}$
at each sensor is equivalent to the ability of locally evaluating
the posterior $\pi_{t\left|\mathscr{Y}_{t}\right.}$, since $\exp\left(-\theta_{t}\left(\boldsymbol{x}\right)/2\right)\equiv\mathsf{L}_{t}\left(\boldsymbol{x}\right)$,
which is the likelihood part of the filter at time $t$. There are
exactly $L$ possibilities for $X_{t}$ and, as a result, in order
to make the computation of the likelihood matrix $\boldsymbol{\Lambda}_{t}$
locally feasible, each sensor should be able to acquire $\mathsf{L}_{t}\left(\boldsymbol{x}_{j}\right)$,
for all $j\in\mathbb{N}_{L}^{+}$. But since there are no functional
dependencies among the latter quantities, each can be separately computed,
in a completely parallel fashion. Consequently, it suffices to focus
exclusively on $\mathsf{L}_{t}\left(\boldsymbol{x}_{j}\right)$, for
some fixed $j\in\mathbb{N}_{L}^{+}$.

Freeze time $t$ temporarily and let $\epsilon>0$. Also, consider
a \textit{symmetric (doubly) stochastic} matrix ${\bf S}\in\mathbb{R}^{S\times S}$,
with ${\bf S}\left(k,l\right)$ being non zero if and only if $l\in\mathsf{N}_{k}$,
for all $k\in\mathbb{N}_{S}^{+}$; precise selection of ${\bf S}$
will be considered later (see Theorem \ref{Optimal_Rates} and relevant
discussion). Then, employing the ADMM based, distributed averaging
\textit{Method B} as discussed in \cite{Consensus_Emiliano_2011},
with $\epsilon{\bf S}$ as the matrix of augmentation constants \cite{Consensus_Emiliano_2011}
and exploitting its symmetricity, it can be easily shown that the
ADMM iterations \cite{Consensus_Emiliano_2011} for the distributed
computation of the average \eqref{eq:CONSENSUS_1}, at sensor $k\in\mathbb{N}_{S}^{+}$,
are reduced to
\begin{flalign}
\boldsymbol{\vartheta}_{t}^{j}\left(n,k\right) & \equiv\dfrac{1}{1+\epsilon}\theta_{t}^{k}\left(\boldsymbol{x}_{j}\right)+\ell_{t}\left(n-1,k\right)+\dfrac{\epsilon}{\left(1+\epsilon\right)}\sum_{l\in\mathsf{N}_{k}}{\bf S}\left(k,l\right)\boldsymbol{\vartheta}_{t}^{j}\left(n-1,l\right)\quad\text{and}\\
\ell_{t}\left(n,k\right) & \equiv\ell_{t}\left(n-1,k\right)+\dfrac{\epsilon}{2\left(1+\epsilon\right)}\hspace{-2pt}\left(\sum_{l\in\mathsf{N}_{k}}\hspace{-2pt}{\bf S}\left(k,l\right)\boldsymbol{\vartheta}_{t}^{j}\left(n\hspace{-2pt}-\hspace{-2pt}1,l\right)-\boldsymbol{\vartheta}_{t}^{j}\left(n\hspace{-2pt}-\hspace{-2pt}1,k\right)\hspace{-2pt}\right)\hspace{-2pt},
\end{flalign}
for all $n\in\mathbb{N}^{2}$. The scheme is initialized as
\begin{flalign}
\boldsymbol{\vartheta}_{t}^{j}\left(1,k\right) & \triangleq\dfrac{1}{1+\epsilon}\theta_{t}^{k}\left(\boldsymbol{x}_{j}\right)\quad\text{and}\\
\ell_{t}\left(1,k\right) & \triangleq0,\quad\forall k\in\mathbb{N}_{S}^{+}.
\end{flalign}
Now, letting $t$ vary in $\mathbb{N}$, at each iteration $n_{t}$
and at each sensor $k$, the true posterior measure $\pi_{t\left|\mathscr{Y}_{t}\right.}$
is approximated as
\begin{equation}
\widetilde{\pi}_{t\left|\mathscr{Y}_{t}\right.}^{k}\left(n_{t}\right)\triangleq\dfrac{\widetilde{\boldsymbol{E}}_{t}^{k}\left(n_{t}\right)}{\left\Vert \widetilde{\boldsymbol{E}}_{t}^{k}\left(n_{t}\right)\right\Vert _{1}},\quad\forall t\in\mathbb{N},
\end{equation}
where the process $\widetilde{\boldsymbol{E}}_{t}^{k}\left(n_{t}\right)\in\mathbb{R}^{L\times1}$
satisfies the linear recursion $\widetilde{\boldsymbol{E}}_{t}^{k}\left(n_{t}\right)\equiv\widetilde{\boldsymbol{\Lambda}}_{t}^{k}\left(n_{t}\right)\boldsymbol{P}\widetilde{\boldsymbol{E}}_{t-1}^{k}\left(n_{t-1}\right),$
for all $t\in\mathbb{N}$, with $\widetilde{\boldsymbol{E}}_{-1}^{k}\left(n_{t}\right)\equiv\widetilde{\boldsymbol{E}}_{-1}^{k}\triangleq\pi_{-1}\equiv\boldsymbol{E}_{-1}$
and where $\widetilde{\boldsymbol{\Lambda}}_{t}^{k}\left(n_{t}\right)$
is defined as in \eqref{eq:LIKELIHOODS_1}, but with the likelihoods
$\mathsf{L}_{t}\left(\boldsymbol{x}_{j}\right)\hspace{-2pt},j\hspace{-2pt}\in\hspace{-2pt}\mathbb{N}_{L}^{+}$
being replaced by
\begin{equation}
\widetilde{\mathsf{L}}_{t}^{n_{t},k}\left(\boldsymbol{x}_{j}\right)\triangleq\exp\left(-\dfrac{\boldsymbol{\vartheta}_{t}^{j}\left(n_{t},k\right)}{2}\right),\quad\forall j\in\mathbb{N}_{L}^{+},
\end{equation}
completing the algorithmic description of the distributed HMM under
consideration. In the above, $n_{t}$ denotes the \textit{iteration
index at time} $t$. In this paper, mainly for analytical and intuitional
simplicity, we will assume that $n_{t}\equiv n,$ for all $t\in\mathbb{N}$.
This simply means that each estimate $\widetilde{\boldsymbol{E}}_{t}^{k}\left(n\right)$
is recursively constructed using the \textit{``same-iteration''}
estimate $\widetilde{\boldsymbol{E}}_{t-1}^{k}\left(n\right)$, \textit{at
the previous time step}. See Fig. \ref{fig:Diagram} for a schematic
representation of the procedure explained above, for each sensor in
the DNA under consideration.
\begin{figure}
\centering\includegraphics[clip]{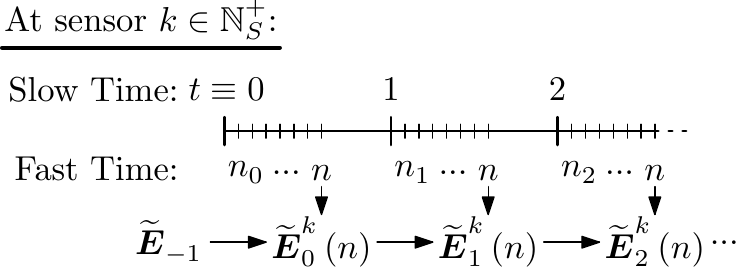}\caption{\label{fig:Diagram}An intuitive schematic representation of the distributed
filtering procedure at sensor $k\in\mathbb{N}_{S}^{+}$.}
\end{figure}

\begin{rem}
For practical considerations, at this point, it would be worth specifying
both the computational complexity and the information exchange complexity
of the distributed filtering considered, \textit{per sensor $k\in\mathbb{N}_{S}^{+}$,
per slow time} $t\in\mathbb{N}$. With regards to computational complexity,
it is easy to see that, at sensor $k$, each consensus iteration incurs
an order of ${\cal O}\left(L\left|\mathsf{N}_{k}\right|\right)$ operations
of combined multiply-adds. Therefore, multiplying by the total number
of consensus iterations per slow time, $n$ (the fast time), and adding
the computational complexity incurred by filtering, results in total
computational complexity at sensor $k$, for each slow time, of the
order of ${\cal O}\left(nL\left|\mathsf{N}_{k}\right|+L^{2}\right)$.
One can see that, for a hidden chain with state space of large cardinality,
$L$, if $n\left|\mathsf{N}_{k}\right|\approx L$, then the computational
complexity of the distributed filtering scheme under consideration
is of the same order as that of filtering alone. As far as information
exchange complexity is concerned, following a similar procedure, it
can be readily shown that an order of ${\cal O}\left(n\left|\mathsf{N}_{k}\right|\right)$
bidirectional (at worst) message exchanges have to take place at sensor
$k$, per slow time $t$.\hfill{}\ensuremath{\blacksquare}
\end{rem}
Arguing as in \cite{Consensus_Emiliano_2011}, it is then straightforward
to show that of central importance concerning the convergence of the
ADMM scheme is the eigenstructure and, in particular, the Second Largest
Eigenvalue Modulus (SLEM) \cite{Boyd2004Fastest} of the matrix \renewcommand{\arraystretch}{1.5}
\begin{equation}
\boldsymbol{M}\triangleq\begin{bmatrix}\dfrac{\epsilon}{1+\epsilon}{\bf S}+{\bf I} & -\dfrac{\epsilon}{2\left(1+\epsilon\right)}\left({\bf S}+{\bf I}\right)\\
{\bf I} & {\bf 0}
\end{bmatrix}\in\mathbb{R}^{2S\times2S}.
\end{equation}
\renewcommand{\arraystretch}{1}Hereafter, let the sets $\left\{ \lambda_{{\bf S}}^{i}\right\} _{i\in\mathbb{N}_{S}^{+}}$
and $\left\{ \lambda_{\boldsymbol{M}}^{i}\right\} _{i\in\mathbb{N}_{2S}^{+}}$
contain the eigenvalues of ${\bf S}$ and $\boldsymbol{M}$, respectively.
Then, recalling that, due to symmetricity, the spectrum of ${\bf S}$
is real, set $\lambda_{{\bf S}}^{1}\triangleq1$ and $\lambda_{{\bf S}}^{2}\triangleq\max_{i\in\left\{ 2,\ldots,S\right\} }\lambda_{{\bf S}}^{i}$.
Because the zero pattern of ${\bf S}$ is the same as that of the
adjacency matrix of the connected RGG modeling the connectivity of
the DNA under consideration (with ones in its diagonal entries), it
can be easily shown that ${\bf S}$ corresponds to the transition
matrix of a Markov chain, which is both irreducible and aperiodic.
Therefore, the SLEM of ${\bf S}$ is strictly smaller than one, and,
consequently, $\lambda_{{\bf S}}^{2}<1$ as well \cite{Boyd2004Fastest}.
Likewise, since $\boldsymbol{M}$ has a unique unit eigenvalue \cite{Consensus_Emiliano_2011},
let $\lambda_{\boldsymbol{M}}^{1}\equiv\left|\lambda_{\boldsymbol{M}}^{1}\right|\triangleq1$,
and also let $0<\rho<1$ (also see \cite{Consensus_Emiliano_2011})
denote the SLEM of $\boldsymbol{M}$, that is, $\rho\triangleq\max_{i\in\left\{ 2,\ldots,2S\right\} }\left|\lambda_{\boldsymbol{M}}^{i}\right|$.
For later reference, define the auxiliary vector $\boldsymbol{\vartheta}_{t}^{j}\left(n\right)\hspace{-2pt}\triangleq\hspace{-2pt}\left[\boldsymbol{\vartheta}_{t}^{j}\left(n,1\right)\,\ldots\,\boldsymbol{\vartheta}_{t}^{j}\left(n,S\right)\right]^{\boldsymbol{T}}\hspace{-2pt}\hspace{-2pt}\in\hspace{-2pt}\mathbb{R}^{S\times1}$.When
applied to the average consensus problem, and specifically to the
distributed filtering problem under consideration, the rate of convergence
of the ADMM is characterized as follows.
\begin{thm}
\textbf{\textup{(ADMM Consensus Error Bound \cite{Consensus_Emiliano_2011})
}}\label{Lemma_ADMM_1} Fix a natural $T>0$, denoting a finite operational
horizon. For each $t\in\mathbb{N}_{T}^{+}$ and an arbitrary $j\in\mathbb{N}_{L}^{+}$,
it is true that, for all $n\in\mathbb{N}^{2}\cap\mathbb{N}^{\left\lfloor \epsilon+1\right\rfloor }$,
\begin{equation}
\left\Vert \boldsymbol{\vartheta}_{t}^{j}\left(n\right)-\boldsymbol{\vartheta}_{t}^{j}\left(\infty\right)\right\Vert _{2}\le\left(1+\epsilon\right)^{-1}\left\Vert \boldsymbol{\theta}_{t}\left(\boldsymbol{x}_{j}\right)\right\Vert _{2}n\rho^{n-1},\label{eq:Bound_1}
\end{equation}
where $\boldsymbol{\vartheta}_{t}^{j}\left(\infty\right)\triangleq\theta_{t}\left(\boldsymbol{x}_{j}\right){\bf 1}_{S}$.
\end{thm}
An apparent conclusion of Theorem \ref{Lemma_ADMM_1} is that the
rate of convergence of the distributed averaging procedure depends
to a large extent on $\epsilon$ and of course the SLEM of $\boldsymbol{M}$,
$\rho$, which, by construction, is also a function of $\epsilon$.
On the other hand, $\rho$ depends on the eigenstructure of ${\bf S}$,
also by construction. As it turns out, the full spectrum of $\boldsymbol{M}$
is explicitly related to the spectrum of ${\bf S}$. The relevant
result follows, complementing the analysis previously presented in
\cite{Consensus_Emiliano_2011}.
\begin{thm}
\textbf{\textup{(Characterization of the Spectrum of $\boldsymbol{M}$)
}}\label{Spectrum_of_M} Given $\lambda_{{\bf S}}^{i},i\in\mathbb{N}_{S}^{+}$
and for any $\epsilon>0$, the eigenvalues of $\boldsymbol{M}$ may
be expressed in pairs as
\begin{equation}
\lambda_{\boldsymbol{M}}^{2i-1,2i}\hspace{-2pt}\left(\lambda_{{\bf S}}^{i},\epsilon\right)\hspace{-2pt}\equiv\hspace{-2pt}\dfrac{1+\epsilon+\epsilon\lambda_{{\bf S}}^{i}\pm\hspace{-2pt}\sqrt{1\hspace{-2pt}+\hspace{-2pt}\epsilon^{2}\left(\left(\lambda_{{\bf S}}^{i}\right)^{2}\hspace{-2pt}\hspace{-2pt}-\hspace{-2pt}1\right)}}{2\left(1+\epsilon\right)},\label{eq:EIGEIG}
\end{equation}
for all $i\in\mathbb{N}_{S}^{+}$. Then, the SLEM of $\boldsymbol{M}$
can be expressed as
\begin{equation}
\rho\left(\epsilon,\lambda_{{\bf S}}^{2}\right)\hspace{-2pt}\equiv\hspace{-2pt}\begin{cases}
\left|\lambda_{\boldsymbol{M}}^{3}\left(\lambda_{{\bf S}}^{2},\epsilon\right)\right|, & \text{if}\text{ }\epsilon\le h\left(\lambda_{{\bf S}}^{2}\right)\\
\dfrac{\epsilon}{1+\epsilon}, & \text{otherwise}
\end{cases}\hspace{-2pt},\;\forall\epsilon>0,\label{eq:SLEM_positive}
\end{equation}
where, for all $\lambda_{{\bf S}}^{2}\in\left[-1,1\right)$,
\begin{equation}
h\left(\lambda_{{\bf S}}^{2}\right)\triangleq\dfrac{1+\lambda_{{\bf S}}^{2}}{1-\lambda_{{\bf S}}^{2}}\mathds{1}_{\left[0,1\right)}\left(\lambda_{{\bf S}}^{2}\right)+\mathds{1}_{\left[-1,0\right)}\left(\lambda_{{\bf S}}^{2}\right).
\end{equation}
Further, for fixed $\lambda_{{\bf S}}^{2}\in\left[-1,1\right)$, $\rho$
is globally minimized at
\begin{equation}
\epsilon^{*}\hspace{-2pt}\left(\lambda_{{\bf S}}^{2}\right)\triangleq\begin{cases}
\dfrac{1}{\sqrt{1-\left(\lambda_{{\bf S}}^{2}\right)^{2}}}, & \text{if}\text{ }\lambda_{{\bf S}}^{2}\in\left[0,1\right)\\
1, & \text{if}\text{ }\lambda_{{\bf S}}^{2}\in\left[-1,0\right)
\end{cases},
\end{equation}
with optimal value
\begin{equation}
\rho^{*}\hspace{-2pt}\left(\lambda_{{\bf S}}^{2}\right)\hspace{-2pt}\equiv\hspace{-2pt}\begin{cases}
\dfrac{\lambda_{{\bf S}}^{2}+1+\sqrt{1\hspace{-2pt}-\hspace{-2pt}\left(\lambda_{{\bf S}}^{2}\right)^{2}}}{2\left(1+\hspace{-2pt}\sqrt{1\hspace{-2pt}-\hspace{-2pt}\left(\lambda_{{\bf S}}^{2}\right)^{2}}\right)}, & \text{if}\text{ }\lambda_{{\bf S}}^{2}\hspace{-2pt}\in\hspace{-2pt}\left[0,1\right)\\
1/2, & \text{if}\text{ }\lambda_{{\bf S}}^{2}\hspace{-2pt}\in\hspace{-2pt}\left[-1,0\right)
\end{cases},\label{eq:Optimal_SLEM}
\end{equation}
constituting a continuous and increasing function of the second largest
eigenvalue of ${\bf S}$, $\lambda_{{\bf S}}^{2}$.
\end{thm}
\begin{proof}[Proof of Theorem \ref{Spectrum_of_M}]
In principle, certain parts of Theorem \ref{Spectrum_of_M} (in particular
\eqref{eq:EIGEIG}) may be derived based on the analysis presented
in \cite{Consensus_Emiliano_2011}. However, the derivation is not
straightforward. For the reader's convenience, a self contained proof
of Theorem \ref{Spectrum_of_M} is presented in Appendix A.
\end{proof}
It will be important to note that all results presented so far regarding
convergence of the ADMM based consensus schema, as well as the eigenstructure
of $\boldsymbol{M}$, hold equally well even if we relax our obvious
demand that the elements of the symmetric (doubly) stochastic matrix
${\bf S}$, corresponding to the non zero entries of the adjacency
matrix of the RGG under consideration, are strictly positive. Such
an assertion can be validated by inspecting the respective proof of
Theorem \ref{Lemma_ADMM_1} in \cite{Consensus_Emiliano_2011} and
observing that the aforementioned assumption on the entries ${\bf S}$
is actually not required; it is the eigenstructure of ${\bf S}$ that
really matters. On the other hand, the property of ${\bf S}$ having
its second largest eigenvalue strictly smaller than unity is not destroyed
by relaxing strict positivity of its elements; the relaxation corresponds
to an enlargement of the space of possible choices for ${\bf S}$
\cite{Boyd2004Fastest}. In addition to the above, it can be readily
verified by Theorem \ref{Spectrum_of_M} that, if $\lambda_{{\bf S}}^{2}<1$,
$\boldsymbol{M}$ will have exactly one unit eigenvalue and that,
under the same condition, $\rho\left(\epsilon,\lambda_{{\bf S}}^{2}\right)<1$,
satisfying the requirement for convergence implied by Theorem \ref{Lemma_ADMM_1}.

All the above will be critical for us, since if we assume that the
diagonal entries of ${\bf S}$ may be nonnegative, then there are
increased degrees of freedom for choosing the best ${\bf S}$, such
that the consensus error bound of Theorem \ref{Lemma_ADMM_1} is optimized,
and there are well established tools for achieving this goal \cite{Boyd_2006Gossip},
as we will discuss below. However, it should be mentioned that mere
nonnegativity of the entries of ${\bf S}$ implies that the consensus
scheme employed for distributed filtering \textit{does not constitute
an ordinary ADMM variant anymore}, since in the usual ADMM based consensus
formulation, the constants multiplying the quadratic part of the augmented
Lagrangian are positive by default. This fact though does not compromise
performance; in fact, it results in a faster distributed averaging
scheme. Therefore, hereafter, we will assume that the zero pattern
of ${\bf S}$ ``includes'' that of the adjacency matrix of the underlying
RGG, \textit{in the sense that all elements of ${\bf S}$ are allowed
to be nonnegative, except for those corresponding to the zero pattern
of the aforementioned} \textit{adjacency matrix}.

Fig. \ref{fig:OPT_SLEM} shows the graph of the optimal SLEM $\rho^{*}\left(\lambda_{{\bf S}}^{2}\right)$
as a function of $\lambda_{{\bf S}}^{2}$, as described in Theorem
\ref{Spectrum_of_M}, as well as the nonoptimized SLEM $\rho\left(\epsilon,\lambda_{{\bf S}}^{2}\right)$,
for a set of fixed suboptimal choices of $\epsilon$'s. The obvious
fact that $\rho^{*}\left(\lambda_{{\bf S}}^{2}\right)$ is a convex,
non affine and increasing function of $\lambda_{{\bf S}}^{2}$ is
very important, because it implies that $\rho^{*}\left(\lambda_{{\bf S}}^{2}\right)$
can be further optimized with respect to the matrix ${\bf S}$, resulting
in substantial performance gain as $\lambda_{{\bf S}}^{2}$ moves
towards zero. In particular, it is true that
\begin{equation}
\min_{{\bf S}\in\mathfrak{S}}\min_{\epsilon\in\mathbb{R}_{++}}\hspace{-2pt}\rho\left(\epsilon,\lambda_{{\bf S}}^{2}\right)\hspace{-2pt}\equiv\min_{{\bf S}\in\mathfrak{S}}\rho^{*}\hspace{-2pt}\left(\lambda_{{\bf S}}^{2}\right)\hspace{-2pt}\equiv\rho^{*}\hspace{-2pt}\left(\min_{{\bf S}\in\mathfrak{S}}\lambda_{{\bf S}}^{2}\right)\hspace{-2pt},\label{eq:OPTIMIZATION_1}
\end{equation}
possibly for different decisions on ${\bf S}$ on the center and right
of \eqref{eq:OPTIMIZATION_1}, where $\mathfrak{S}$ denotes the feasible
set of symmetric (doubly) stochastic matrices, with ${\bf S}\left(k,l\right)$
being zero if $l\notin\mathsf{N}_{k}$, for all $k\in\mathbb{N}_{S}^{+}$.
It is well known that the optimization problem on the right of \eqref{eq:OPTIMIZATION_1}
is convex and that it can be solved efficiently either via semidefinite
programming, or subgradient methods \cite{Boyd_2006Gossip}. Therefore,
the objective on the left of \eqref{eq:OPTIMIZATION_1} can be globally
optimized as well.
\begin{figure}
\centering\includegraphics[clip,scale=0.61]{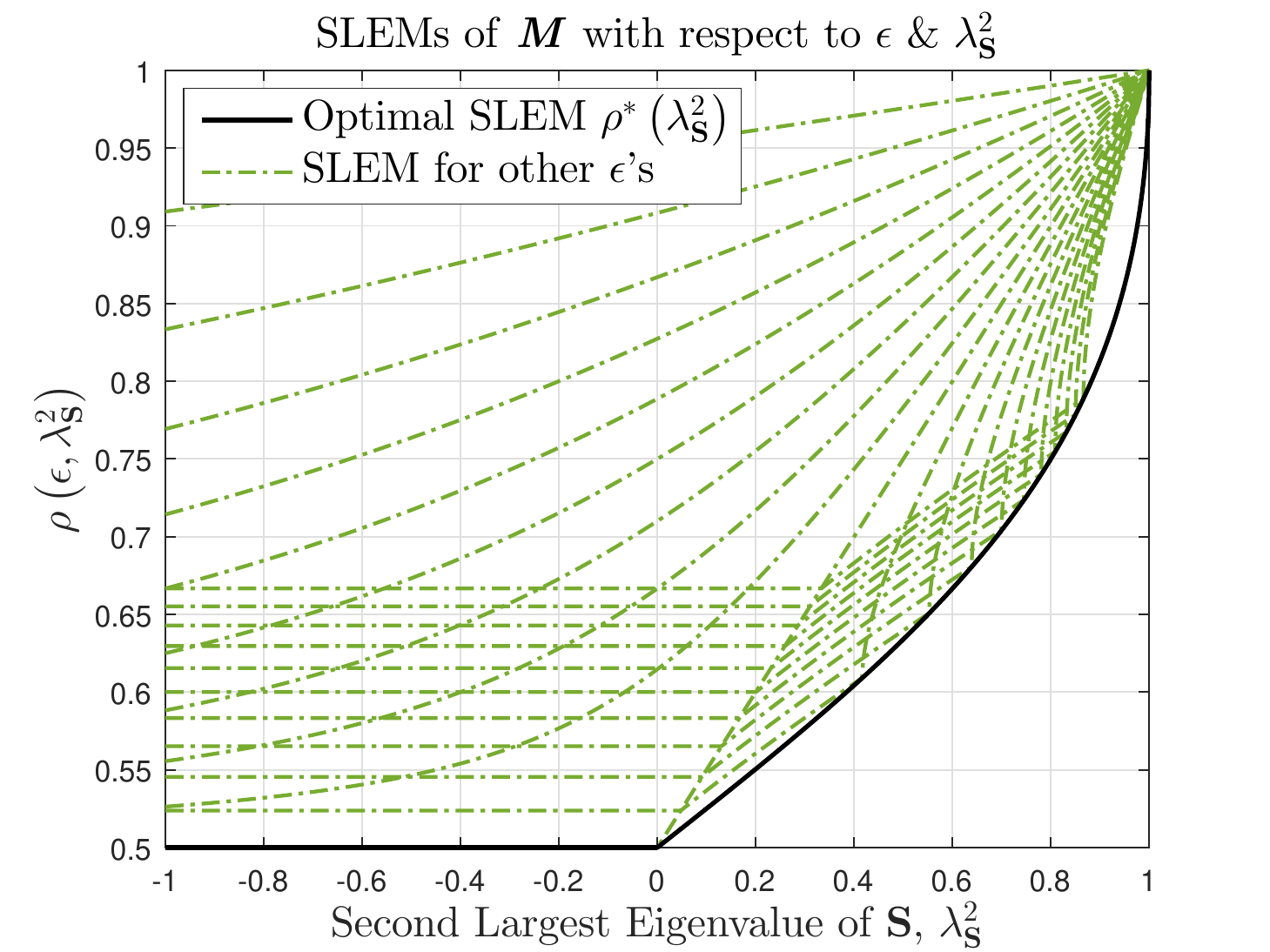}\caption{\label{fig:OPT_SLEM}SLEMs of $\boldsymbol{M}$ as a function of the
largest eigenvalue of ${\bf S}$, $\lambda_{{\bf S}}^{2}$, for various
fixed values of $\epsilon$ (green dash-dot line), as well as the
optimal choice $\epsilon^{*}\hspace{-2pt}\left(\lambda_{{\bf S}}^{2}\right)$
(black solid line).}
\end{figure}

Quite surprisingly, under some additional, albeit very mild, conditions
on the ADMM iteration index, $n$, optimization of the SLEM of $\rho\left(\epsilon,\lambda_{{\bf S}}^{2}\right)$,
either with respect to $\epsilon$ or both $\epsilon$ and ${\bf S}$
indeed implies optimization of the consensus error bound of the ADMM,
as stated in Theorem \ref{Lemma_ADMM_1}, under a certain sense. In
this respect, we present the next result.
\begin{thm}
\textbf{\textup{(ADMM Optimal Consensus Error Bound I) }}\label{Optimal_Rates}
Fix a natural $T>0$ and choose $\epsilon_{max}\ge\epsilon^{*}\hspace{-2pt}\left(\lambda_{{\bf S}}^{2}\right)$,
$\lambda_{{\bf S}}^{2}\in\left[-1,1\right)$, for some ${\bf S}\in\mathfrak{S}$.
For each $t\in\mathbb{N}_{T}^{+}$ and any $j\in\mathbb{N}_{L}^{+}$,
as long as $n>2\epsilon_{max}+1$, the choice $\epsilon\equiv\epsilon^{*}\hspace{-2pt}\left(\lambda_{{\bf S}}^{2}\right)$
is Pareto optimal for the scalar, multiobjective program
\begin{equation}
\underset{\epsilon>0}{\mathrm{minimize}}\hspace{5pt}\left[\left\{ \left(\rho\left(\epsilon,\lambda_{{\bf S}}^{2}\right)\right)^{n-1}\hspace{-2pt}/\left(1+\epsilon\right)\text{ s.t. }\epsilon<n\right\} _{n}\right]\hspace{-2pt},
\end{equation}
optimizing the RHS of \eqref{eq:Bound_1}, resulting in the optimal
consensus error bound 
\begin{equation}
\left\Vert \boldsymbol{\vartheta}_{t}^{j}\left(n\right)\hspace{-2pt}-\hspace{-2pt}\boldsymbol{\vartheta}_{t}^{j}\left(\infty\right)\right\Vert _{2}\hspace{-2pt}\le\hspace{-2pt}\gamma\hspace{-2pt}\left(\lambda_{{\bf S}}^{2}\right)\hspace{-2pt}\left\Vert \boldsymbol{\theta}_{t}\left(\boldsymbol{x}_{j}\right)\right\Vert _{2}n\hspace{-2pt}\left(\rho^{*}\hspace{-2pt}\left(\lambda_{{\bf S}}^{2}\right)\hspace{-2pt}\right)^{\hspace{-2pt}n}\hspace{-2pt},\label{eq:Bound_2}
\end{equation}
where
\begin{equation}
\gamma\hspace{-2pt}\left(\lambda_{{\bf S}}^{2}\right)\hspace{-2pt}\triangleq\hspace{-2pt}\begin{cases}
\dfrac{2\sqrt{1\hspace{-2pt}-\hspace{-2pt}\left(\lambda_{{\bf S}}^{2}\right)^{2}}}{1+\lambda_{{\bf S}}^{2}+\hspace{-2pt}\sqrt{1\hspace{-2pt}-\hspace{-2pt}\left(\lambda_{{\bf S}}^{2}\right)^{2}}}, & \text{if}\text{ }\lambda_{{\bf S}}^{2}\hspace{-2pt}\in\hspace{-2pt}\left[0,1\right)\\
1, & \text{if}\text{ }\lambda_{{\bf S}}^{2}\hspace{-2pt}\in\hspace{-2pt}\left[-1,0\right)
\end{cases},
\end{equation}
being a continuous and decreasing function of $\lambda_{{\bf S}}^{2}$
in $\left[-1,1\right)$. Additionally, choose any $\widetilde{{\bf S}}\in\mathfrak{S}$.
If, under the same setting as above, $\lambda_{\widetilde{{\bf S}}}^{2}$
is replaced by the optimal value $\min_{{\bf S}\in\mathfrak{S}}\lambda_{{\bf S}}^{2}\le\lambda_{\widetilde{{\bf S}}}^{2}$
in \eqref{eq:Bound_2}, then the resulting bound is optimal, in the
sense that
\begin{equation}
\min_{\substack{{\bf S}\in\mathfrak{S}\\
\lambda_{{\bf S}}^{2}\in\left[-1,\lambda_{\widetilde{{\bf S}}}^{2}\right]
}
}\hspace{-2pt}\hspace{-2pt}\hspace{-2pt}\hspace{-2pt}\hspace{-2pt}\hspace{-2pt}\gamma\hspace{-2pt}\left(\lambda_{{\bf S}}^{2}\right)\hspace{-2pt}\hspace{-2pt}\left(\rho^{*}\hspace{-2pt}\left(\lambda_{{\bf S}}^{2}\right)\hspace{-2pt}\right)^{\hspace{-2pt}n}\hspace{-2pt}\hspace{-2pt}\equiv\hspace{-2pt}\gamma\hspace{-2pt}\left(\min_{{\bf S}\in\mathfrak{S}}\lambda_{{\bf S}}^{2}\hspace{-2pt}\right)\hspace{-2pt}\hspace{-2pt}\left(\hspace{-2pt}\rho^{*}\hspace{-2pt}\left(\min_{{\bf S}\in\mathfrak{S}}\lambda_{{\bf S}}^{2}\hspace{-2pt}\right)\hspace{-2pt}\right)^{\hspace{-2pt}n}\hspace{-2pt}\hspace{-1.4pt}.
\end{equation}
\end{thm}
\begin{proof}[Proof of Theorem \ref{Optimal_Rates}]
See Appendix B.
\end{proof}
A graphical demonstration of the second, somewhat technical part of
Theorem \ref{Optimal_Rates} is shown in Fig. \ref{fig:ERROR_BOUNDS}.
As it can be readily observed, given an arbitrary, possibly ``bad''
$\widetilde{{\bf S}}\in\mathfrak{S}$, corresponding to a specific,
possibly large, second largest eigenvalue $\lambda_{\widetilde{{\bf S}}}^{2}$
(in the figure, this value equals $0.995$), the best convergence
rate is attained at $\min_{{\bf S}\in\mathfrak{S}}\lambda_{{\bf S}}^{2}$,
when the feasible set is constrained such that $\lambda_{{\bf S}}^{2}\in\left[-1,\lambda_{\widetilde{{\bf S}}}^{2}\right]$.
Also note that, because $\epsilon^{*}\hspace{-2pt}\left(\lambda_{\widetilde{{\bf S}}}^{2}\right)$
is increasing in $\lambda_{\widetilde{{\bf S}}}^{2}$, Pareto efficiency
of \eqref{eq:Bound_2} is preserved when $\lambda_{\widetilde{{\bf S}}}^{2}$
is replaced by $\min_{{\bf S}\in\mathfrak{S}}\lambda_{{\bf S}}^{2}$.
\begin{figure}
\centering\includegraphics[clip,scale=0.61]{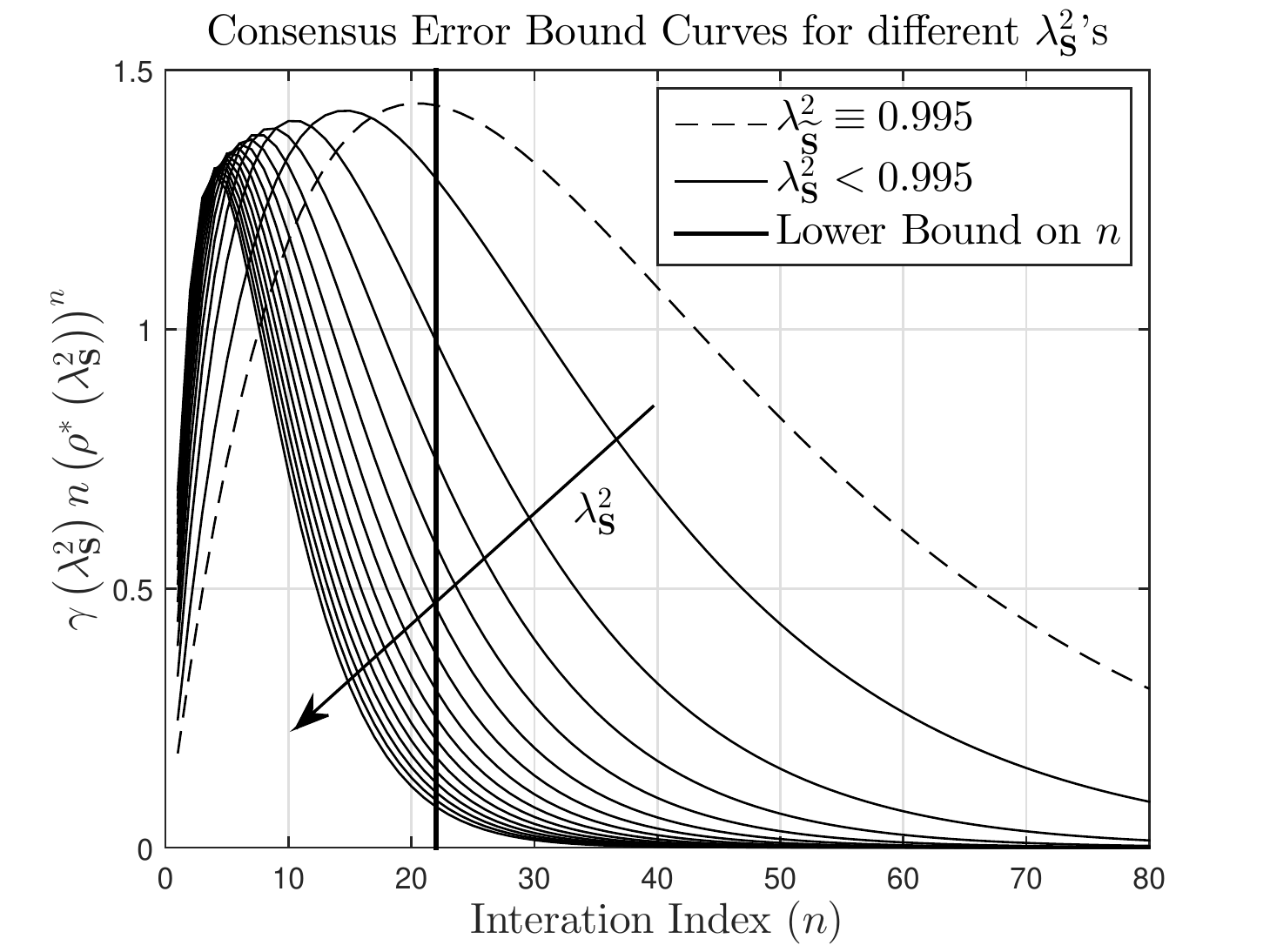}\caption{\label{fig:ERROR_BOUNDS}Demonstration of the second part of Theorem
\ref{Optimal_Rates}: Consensus error bounds for various values of
$\lambda_{{\bf S}}^{2}$, when $\widetilde{{\bf S}}$ is chosen such
that $\lambda_{\widetilde{{\bf S}}}^{2}\equiv0.995$. }
\end{figure}
\begin{figure*}
\centering$\hspace{-11pt}$\subfloat[\label{fig:Channel_Tracking}]{\centering\includegraphics[scale=0.61]{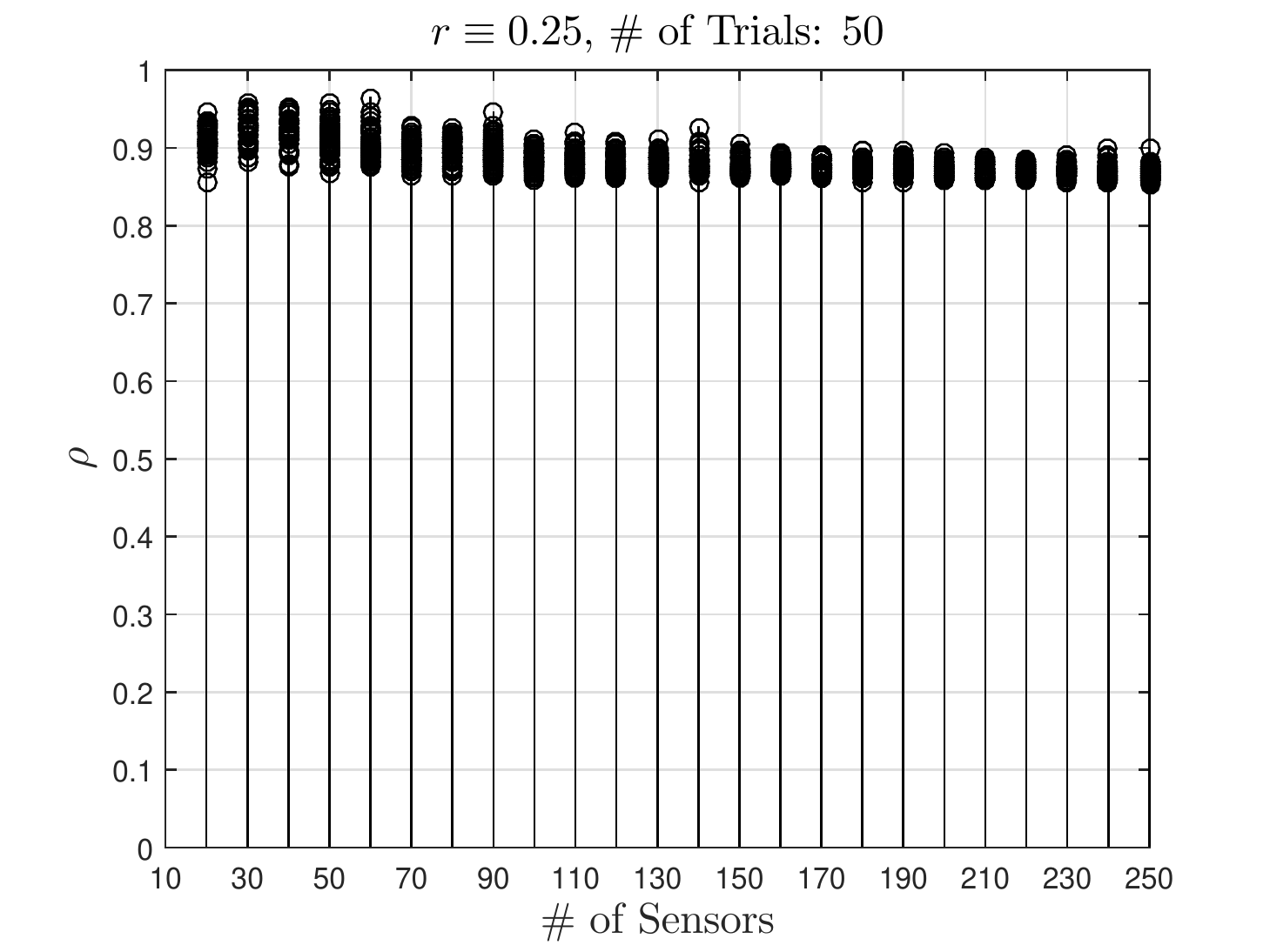}

}$\hspace{-20pt}$\subfloat[\label{fig:Spatial_Pred}]{\centering\includegraphics[clip,scale=0.61]{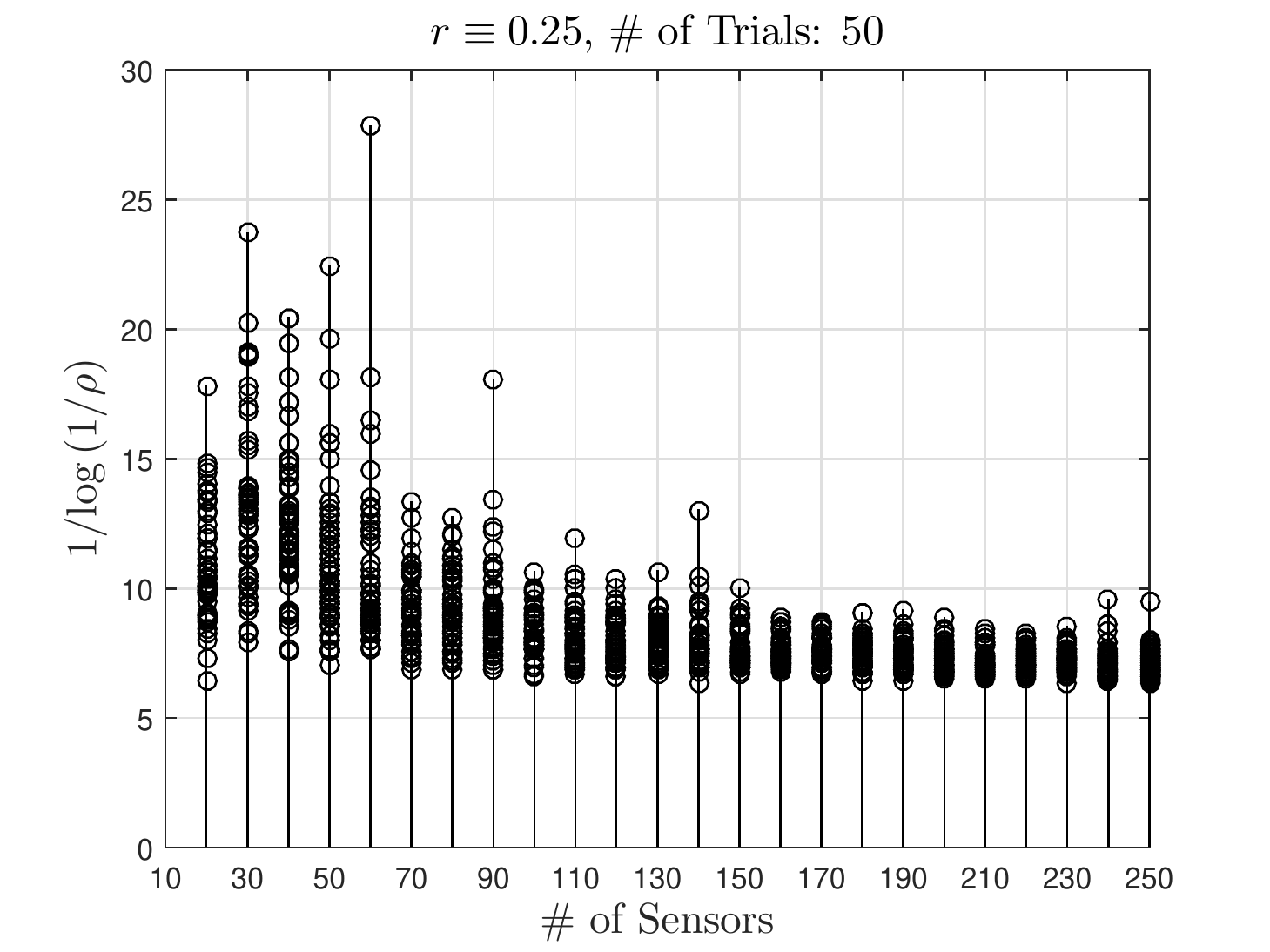}}\caption{\label{fig:MIXING}The quantities (a) $\rho$ and (b) $\tau\equiv1/\log\left(1/\rho\right)$,
as functions of the number of sensors $S$. For each value of $S$,
$50$ independent realizations of the respective quantities are shown.}
\end{figure*}

Some of the technicalities of Theorem \ref{Optimal_Rates} may be
avoided if one focuses on optimizing the looser consensus error bound
\begin{equation}
\left\Vert \boldsymbol{\vartheta}_{t}^{j}\left(n\right)-\boldsymbol{\vartheta}_{t}^{j}\left(\infty\right)\right\Vert _{2}\le\left\Vert \boldsymbol{\theta}_{t}\left(\boldsymbol{x}_{j}\right)\right\Vert _{2}n\rho^{n-1},\label{eq:SIMPLER_Bound_1}
\end{equation}
since $\left(1+\epsilon\right)^{-1}<1$. Indeed, the following result
holds; proof is omitted.
\begin{thm}
\textbf{\textup{(ADMM Optimal Consensus Error Bound II) }}\label{Optimal_Rates-1}Fix
a natural $T>0$ and choose $\epsilon_{max}\ge\epsilon^{*}\hspace{-2pt}\left(\lambda_{\widetilde{{\bf S}}}^{2}\right)$,
$\lambda_{\widetilde{{\bf S}}}^{2}\in\left[-1,1\right)$, for some
$\widetilde{{\bf S}}\in\mathfrak{S}$. For each $t\in\mathbb{N}_{T}^{+}$
and any $j\in\mathbb{N}_{L}^{+}$, the globally optimal consensus
error bound corresponding to the RHS of \eqref{eq:SIMPLER_Bound_1},
with respect to
\begin{equation}
\left(\epsilon,{\bf S}\right)\in\left(0,\epsilon_{max}\right]\hspace{-2pt}\times\hspace{-2pt}\left\{ \boldsymbol{C}\in\mathfrak{S}\left|\lambda_{\boldsymbol{C}}^{2}\in\hspace{-2pt}\left[-1,\lambda_{\widetilde{{\bf S}}}^{2}\right]\right.\hspace{-2pt}\right\} \hspace{-2pt},
\end{equation}
is given by 
\begin{equation}
\left\Vert \boldsymbol{\vartheta}_{t}^{j}\left(n\right)\hspace{-2pt}-\hspace{-2pt}\boldsymbol{\vartheta}_{t}^{j}\left(\infty\right)\right\Vert _{2}\hspace{-2pt}\le\hspace{-2pt}\left\Vert \boldsymbol{\theta}_{t}\hspace{-2pt}\left(\boldsymbol{x}_{j}\right)\right\Vert _{2}\hspace{-1pt}n\hspace{-2pt}\left(\rho^{*}\hspace{-2pt}\left(\min_{{\bf S}\in\mathfrak{S}}\lambda_{{\bf S}}^{2}\right)\hspace{-2pt}\right)^{\hspace{-2pt}n-1}\hspace{-2pt},\label{eq:Bound_2-1}
\end{equation}
for all $n>\max\left\{ 2,\epsilon_{max}\right\} $.
\end{thm}
Finally, in our main results, presented in Section \ref{sec:MAIN_RESULTS},
the quantity $\tau\triangleq1/\log\left(1/\rho\right)$ will be of
great importance in establishing a minimal number of ADMM iterations
required between each pair of subsequent observation times, ensuring
stability of the distributed nonlinear filter under consideration,
over a finite horizon of filtering operation, $T$. Interestingly,
$\tau$ does not appear coincidentally in our analysis; when analyzing
the mixing properties of Markov chains over graphs, $\tau$ constitutes
a known quantity, called the \textit{mixing time} of the chain \cite{Boyd2004Fastest}.
In this fashion, we call $\mbox{\ensuremath{\tau}}$ the mixing time
of $\boldsymbol{M}$, characterizing the convergence rate of the ADMM.

Of course, it is both expected and intuitively correct that $\rho$
and, therefore, $\tau$, will be both dependent on the connectivity
density of the underlying RGG (measured, for instance, through the
sparsity of the adjacency matrix of the graph), modeling the connectivity
of the DNA under study. More specifically, we should at least expect
that $\rho$ (which constitutes a random variable) will be in general
smaller in ``more strongly'' connected RGGs, and the same for $\tau$.
Also, since more sensors uniformly scattered in the same area mean
essentially higher connectivity density, the same behavior is expected
as the number of sensors increase. For example, consider the case
where we are given a DNA corresponding to a \textit{strongly} connected
RGG. Unfortunately though, an analytic characterization of the aforementioned
properties boils down to analyzing the relevant properties of $\lambda_{{\bf S}}^{2}$,
which constitutes a notoriously difficult (and still open) mathematical
problem. Despite of this difficulty, we can verify the aforementioned
behavior experimentally. Indeed, Fig. \ref{fig:MIXING} shows $\tau$
(a) and $\rho$ (b) as functions of the number of sensors, $S$. For
each value of $S$, $50$ trials are presented. In this example, ${\bf S}$
is suboptimally chosen according to the \textit{maximum degree chain}
construction \cite{Boyd2004Fastest}, whereas $\epsilon$ is optimally
chosen. We readily observe that our intuitive assertions stated above
are successfully verified. Based on these results, in all the subsequent
analysis, we will consider $\tau$ and $\rho$ to be \textit{bounded
and in general decreasing} functions of the number of sensors in the
network, $S$, and/or the connectivity threshold, $r$. This fact
constitutes the direct conclusion of Fig. \ref{fig:MIXING}. It is
not straightforward to determine how the actual performance of the
whole distributed filtering scheme, is affected by $S$, the threshold
$r$ and the corresponding values of $\rho$ and $\tau$. This will
be precisely the topic of interest in the next section, where we show
that, in order to guarantee $\varepsilon$-global agreement on filtering
estimates, it is sufficient that the number of consensus iterations
per sensor, per slow time, $n$, scales linearly with $\tau$ and,
at worst, loglinearly with $S$ (note that, on average, $S$ is inversely
proportional to $\tau$, as discussed above).

\section{\label{sec:MAIN_RESULTS}Uniform Stability of Distributed Filtering}

This section is devoted to the presentation and proof of the main
results of the paper. Hereafter, for notational brevity, we will set
$\gamma\hspace{-2pt}\left(\lambda_{{\bf S}}^{2}\right)\equiv\gamma$
and $\rho^{*}\hspace{-2pt}\left(\lambda_{{\bf S}}^{2}\right)\equiv\rho$,
in accordance to Theorem \ref{Optimal_Rates} presented in Section
\ref{sec:Preliminaries}. Therefore, it will also be assumed that
everything holds as long as $n>2\epsilon_{max}+1$, on top of any
other condition on the iteration index, $n$. If, for any reason,
one wishes to consider Theorems \ref{Lemma_ADMM_1} or \ref{Optimal_Rates-1},
the aforementioned statements have to be modified accordingly.

The following lemmata, borrowed from \cite{KalPetNonlinear_2015},
will be helpful in the subsequent analysis.
\begin{lem}
\textbf{\textup{(Telescoping Bound for Matrix Products \cite{KalPetNonlinear_2015})}}\label{Elementary_Lemma_2}
Consider the collections of arbitrary, square matrices
\begin{flalign*}
\left\{ {\bf A}_{i}\in\mathbb{C}^{N\times N}\right\} _{i\in\mathbb{N}_{n}}\quad\text{and}\quad & \left\{ {\bf B}_{i}\in\mathbb{C}^{N\times N}\right\} _{i\in\mathbb{N}_{n}}.
\end{flalign*}
For any submultiplicative matrix norm $\left\Vert \cdot\right\Vert _{\mathfrak{M}}$,
it is true that\footnote{Hereafter, we adopt the conventions $\prod_{j=0}^{-1}\left(\cdot\right)\equiv\prod_{j=n+1}^{n}\left(\cdot\right)\triangleq1$. }
\begin{align}
\left\Vert \prod_{i=0}^{n}{\bf A}_{i}-\prod_{i=0}^{n}{\bf B}_{i}\right\Vert _{\mathfrak{M}} & \le\sum_{i=0}^{n}\left(\prod_{j=0}^{i-1}\left\Vert {\bf A}_{j}\right\Vert _{\mathfrak{M}}\right)\left(\prod_{j=i+1}^{n}\left\Vert {\bf B}_{j}\right\Vert _{\mathfrak{M}}\right)\left\Vert {\bf A}_{i}-{\bf B}_{i}\right\Vert _{\mathfrak{M}}.\label{eq:Product_Inequality}
\end{align}
\end{lem}

\begin{lem}
\textbf{\textup{(Stability of Observations \cite{KalPetNonlinear_2015})}}
\label{Lemma_WHP} Consider the random quadratic form
\begin{align}
\hspace{-2pt}\hspace{-2pt}\hspace{-2pt}Q_{t}\hspace{-2pt}\left(\omega\right) & \hspace{-2pt}\triangleq\hspace{-2pt}\left\Vert {\bf y}_{t}\hspace{-2pt}\left(\omega\right)\right\Vert _{2}^{2}\hspace{-2pt}\equiv\hspace{-2pt}\left\Vert \overline{{\bf y}}_{t}\hspace{-2pt}\left(X_{t}\left(\omega\right)\right)\hspace{-2pt}+\hspace{-2pt}\boldsymbol{\mu}_{t}\hspace{-2pt}\left(X_{t}\left(\omega\right)\right)\right\Vert _{2}^{2}\hspace{-2pt},\,t\in\mathbb{N}.
\end{align}
Then, for any fixed $t\in\mathbb{N}$ and any freely chosen $C\ge1$,
there exists a bounded constant $\beta>1$, such that the measurable
set
\begin{equation}
{\cal T}_{t}\triangleq\left\{ \omega\in\Omega\left|\sup_{i\in\mathbb{N}_{t}}Q_{i}\left(\omega\right)<\beta CN\left(1+\log\left(t+1\right)\right)\right.\right\} 
\end{equation}
satisfies
\begin{equation}
{\cal P}\left({\cal T}_{t}\right)\ge1-\left(t+1\right)^{1-CN}\exp\left(-CN\right),
\end{equation}
that is, the sequence of quadratic forms $\left\{ Q_{i}\left(\omega\right)\right\} _{i\in\mathbb{N}_{t}}$
is uniformly bounded with overwhelmingly high probability, under the
probability measure ${\cal P}$.
\end{lem}
The next result constitutes a corollary of Theorem \ref{Lemma_ADMM_1}.
Here, for our purposes, we state it as a preliminary lemma.
\begin{lem}
\textbf{\textup{(Rate of each Iterate)}} \label{Lemma_ADMM_2} For
each $t\in\mathbb{N}_{T}^{+}$ and an arbitrary $j\in\mathbb{N}_{L}^{+}$,
it is true that, for all $n\in\mathbb{N}^{\left\lfloor 2\epsilon_{max}+2\right\rfloor }$,
\begin{equation}
\left|\boldsymbol{\vartheta}_{t}^{j}\left(n,k\right)-\theta_{t}\left(\boldsymbol{x}_{j}\right)\right|\le\gamma\left\Vert \boldsymbol{\theta}_{t}\left(\boldsymbol{x}_{j}\right)\right\Vert _{2}n\rho^{n},
\end{equation}
for all $k\in\mathbb{N}_{S}^{+}$.
\end{lem}
\begin{proof}[Proof of Lemma \ref{Lemma_ADMM_2}]
See Appendix C.
\end{proof}
Employing Lemma \ref{Lemma_WHP}, we can characterize the growth of
the the initial vector $\boldsymbol{\theta}_{t}\left(\boldsymbol{x}_{j}\right)$,
for $t\in\mathbb{N}_{T}$, as follows.
\begin{lem}
\textbf{\textup{(Growth of Initial Values)}} \label{Lemma_GROWTH}
There exists a $\delta>1$, such that 
\begin{flalign}
\sup_{t\in\mathbb{N}_{T}^{+}}\sup_{j\in\mathbb{N}_{L}^{+}}\left\Vert \boldsymbol{\theta}_{t}\left(\boldsymbol{x}_{j}\right)\right\Vert _{2} & \le\delta S^{1.5}CN\left(1+\log\left(T+1\right)\right),
\end{flalign}
with probability at least $1-\left(T+1\right)^{1-CN}\exp\left(-CN\right)$,
for any free $C\ge1$.
\end{lem}
\begin{proof}[Proof of Lemma \ref{Lemma_GROWTH}]
See Appendix D.
\end{proof}
Let us now present another lemma, providing a probabilistic uniform
lower bound concerning the iterates $\boldsymbol{\vartheta}_{t}^{j}\left(n,k\right)$,
under appropriate conditions on the number of iterations, $n$.
\begin{lem}
\textbf{\textup{(A lower bound on each Iterate)}} \label{Lemma_ADMM_3}
It is true that, as long as $n\in\mathbb{N}^{\left\lfloor 2\epsilon_{max}+2\right\rfloor }$
and
\begin{equation}
n-\tau\log\left(\gamma n\right)\hspace{-2pt}\ge\hspace{-2pt}\tau\log\left(\dfrac{2\delta S^{1.5}C\left(1+\log\left(T+1\right)\right)}{\log\left(\lambda_{inf}\right)}\right),\label{eq:Lower_Bound_1}
\end{equation}
the iterates generated by the ADMM are lower bounded as
\begin{equation}
\inf_{t\in\mathbb{N}_{T}^{+}}\inf_{j\in\mathbb{N}_{L}^{+}}\inf_{k\in\mathbb{N}_{S}^{+}}\boldsymbol{\vartheta}_{t}^{j}\left(n,k\right)\hspace{-2pt}\ge\hspace{-2pt}\dfrac{N}{2}\log\left(\lambda_{inf}\right),
\end{equation}
with probability at least $1-\left(T+1\right)^{1-CN}\exp\left(-CN\right)$.
\end{lem}
\begin{proof}[Proof of Lemma \ref{Lemma_ADMM_3}]
See Appendix E.
\end{proof}
Leveraging the results presented above, we may now provide a probabilistic
bound on the required number of ADMM iterations, such that the $\ell_{1}$
norm between the \textbf{\textit{unnormalized}} filtering estimates
$\boldsymbol{E}_{t}$ and $\widetilde{\boldsymbol{E}}_{t}^{k}\left(n\right)$
is small.
\begin{thm}
\textbf{\textup{(Stability of Unnormalized Distributed Estimates)}}
\label{Theorem_Filtering_ADMM} Fix $T<\infty$ and choose any global
accuracy level $0<\varepsilon<1$. Then, there exists $\eta>1$ such
that, as long as $n\in\mathbb{N}^{\left\lfloor 2\epsilon_{max}+2\right\rfloor }$
and
\begin{equation}
n-\tau\log\left(\gamma n\right)\ge\tau\log\left(\dfrac{\eta S^{1.5}C\left(1+\log\left(T+1\right)\right)}{\varepsilon}\right),\label{eq:CONDITION}
\end{equation}
the absolute error between $\boldsymbol{E}_{t}$ and $\widetilde{\boldsymbol{E}}_{t}^{k}\left(n\right)$
satisfies
\begin{equation}
\sup_{t\in\mathbb{N}_{T}}\sup_{k\in\mathbb{N}_{S}^{+}}\left\Vert \boldsymbol{E}_{t}-\widetilde{\boldsymbol{E}}_{t}^{k}\left(n\right)\right\Vert _{1}\le\varepsilon,
\end{equation}
with probability at least $1-\left(T+1\right)^{1-CN}\exp\left(-CN\right)$.
That is, provided \eqref{eq:CONDITION} holds locally at each sensor
$k\in\mathbb{N}_{S}^{+}$, at each $t\in\mathbb{N}_{T}$, $\boldsymbol{E}_{t}$
equals $\widetilde{\boldsymbol{E}}_{t}^{k}\left(n\right)$ \textbf{within}
$\varepsilon$, with overwhelmingly high probability.
\end{thm}
\begin{proof}[Proof of Theorem \ref{Theorem_Filtering_ADMM}]
First, we know that $\boldsymbol{E}_{t}\equiv\boldsymbol{\Lambda}_{t}\boldsymbol{P}\boldsymbol{E}_{t-1}$,
for all $t\in\mathbb{N}$, where $\boldsymbol{E}_{-1}\equiv\pi_{-1}$.
Therefore, simple induction shows that
\begin{equation}
\boldsymbol{E}_{t}=\left(\prod_{i\in\mathbb{N}_{t}}\left(\boldsymbol{\Lambda}_{t-i}\boldsymbol{P}\right)\right)\boldsymbol{E}_{-1},\quad\forall t\in\mathbb{N}.
\end{equation}
Likewise, by construction of our distributed state estimator,
\begin{equation}
\widetilde{\boldsymbol{E}}_{t}^{k}\left(n\right)=\left(\prod_{i\in\mathbb{N}_{t}}\left(\widetilde{\boldsymbol{\Lambda}}_{t-i}^{k}\left(n\right)\boldsymbol{P}\right)\right)\widetilde{\boldsymbol{E}}_{-1}^{k},\quad\forall t\in\mathbb{N}
\end{equation}
and for all $k\in\mathbb{N}_{S}^{+}$, where $\widetilde{\boldsymbol{E}}_{-1}^{k}\equiv\boldsymbol{E}_{-1}$,
identically. Taking the difference between $\boldsymbol{E}_{t}$ and
$\widetilde{\boldsymbol{E}}_{t}^{k}\left(n\right)$, we can write
\begin{align}
\boldsymbol{E}_{t}\hspace{-2pt}-\hspace{-2pt}\widetilde{\boldsymbol{E}}_{t}^{k}\hspace{-2pt}\left(n\right) & \hspace{-2pt}=\hspace{-2pt}\hspace{-2pt}\left(\prod_{i\in\mathbb{N}_{t}}\hspace{-2pt}\left(\boldsymbol{\Lambda}_{t-i}\boldsymbol{P}\right)\hspace{-2pt}-\hspace{-2pt}\hspace{-2pt}\prod_{i\in\mathbb{N}_{t}}\hspace{-2pt}\left(\widetilde{\boldsymbol{\Lambda}}_{t-i}^{k}\left(n\right)\boldsymbol{P}\right)\hspace{-2pt}\right)\hspace{-2pt}\boldsymbol{E}_{-1}.\label{eq:help!}
\end{align}
Then, taking the $\ell_{1}$-norm on both sides of \eqref{eq:help!},
using the fact that $\left\Vert \boldsymbol{E}_{-1}\right\Vert _{1}\equiv1$
and invoking Lemma \ref{Elementary_Lemma_2}, we get
\begin{align}
 & \hspace{-2pt}\hspace{-2pt}\hspace{-2pt}\hspace{-2pt}\hspace{-2pt}\hspace{-2pt}\hspace{-2pt}\hspace{-2pt}\left\Vert \boldsymbol{E}_{t}-\widetilde{\boldsymbol{E}}_{t}^{k}\left(n\right)\right\Vert _{1}\nonumber \\
 & \le\sum_{i\in\mathbb{N}_{t}}\left(\prod_{j=0}^{i-1}\left\Vert \boldsymbol{\Lambda}_{t-j}\right\Vert _{1}\left\Vert \boldsymbol{P}\right\Vert _{1}\right)\left(\prod_{j=i+1}^{t}\left\Vert \widetilde{\boldsymbol{\Lambda}}_{t-j}^{k}\left(n\right)\right\Vert _{1}\left\Vert \boldsymbol{P}\right\Vert _{1}\right)\left\Vert \left(\boldsymbol{\Lambda}_{t-i}-\widetilde{\boldsymbol{\Lambda}}_{t-i}^{k}\left(n\right)\right)\right\Vert _{1}\left\Vert \boldsymbol{P}\right\Vert _{1}.\label{eq:Theorem_1}
\end{align}
Of course, the $\ell_{1}$-norm for a matrix (induced by the usual
$\ell_{1}$-norm for vectors) coincides with its maximum column sum.
As a result, it will be true that $\left\Vert \boldsymbol{P}\right\Vert _{1}\equiv1$,
since $\boldsymbol{P}$ constitutes a column stochastic matrix. Now,
recalling the internal structure of the diagonal matrices $\boldsymbol{\Lambda}_{t}$
and $\widetilde{\boldsymbol{\Lambda}}_{t}^{k}$, $t\in\mathbb{N}$,
the fact that $\lambda_{inf}$ is strictly greater than unity and
known apriori and Lemma \ref{Lemma_ADMM_3} stated and proved above,
it will be true that
\begin{flalign}
\left\Vert \boldsymbol{\Lambda}_{t-i}\right\Vert _{1} & \le\dfrac{1}{\lambda_{inf}^{N/2}}\quad\text{and}\quad\left\Vert \widetilde{\boldsymbol{\Lambda}}_{t-i}^{k}\left(n\right)\right\Vert _{1}\le\dfrac{1}{\lambda_{inf}^{N/4}},
\end{flalign}
for all $i\in\mathbb{N}_{t}$, under the conditions the aforementioned
lemma suggests. Then, the RHS of \eqref{eq:Theorem_1} can be further
bounded as
\begin{alignat}{1}
\left\Vert \boldsymbol{E}_{t}-\widetilde{\boldsymbol{E}}_{t}^{k}\left(n\right)\right\Vert _{1} & \hspace{-2pt}\le\hspace{-2pt}\sum_{i\in\mathbb{N}_{t}}\dfrac{\left\Vert \left(\boldsymbol{\Lambda}_{t-i}-\widetilde{\boldsymbol{\Lambda}}_{t-i}^{k}\left(n\right)\right)\right\Vert _{1}}{\lambda_{inf}^{Ni/2}\lambda_{inf}^{N\left(t-i\right)/4}}\nonumber \\
 & \hspace{-2pt}\equiv\hspace{-2pt}\dfrac{1}{\lambda_{inf}^{Nt/4}}\hspace{-2pt}\sum_{i\in\mathbb{N}_{t}}\hspace{-2pt}\dfrac{\left\Vert \left(\boldsymbol{\Lambda}_{t-i}-\widetilde{\boldsymbol{\Lambda}}_{t-i}^{k}\left(n\right)\right)\right\Vert _{1}}{\lambda_{inf}^{Ni/4}}.\label{eq:Theorem_2}
\end{alignat}
Next, focusing on each term inside the summation above, we have (see
Lemma \ref{Lemma_GROWTH})
\begin{flalign}
\left\Vert \left(\boldsymbol{\Lambda}_{t-i}-\widetilde{\boldsymbol{\Lambda}}_{t-i}^{k}\left(n\right)\right)\right\Vert _{1} & \hspace{-2pt}\equiv\hspace{-2pt}\max_{l\in\mathbb{N}_{L}^{+}}\left|\exp\left(-\dfrac{1}{2}\theta_{t-i}\left(\boldsymbol{x}_{j}\right)\right)-\exp\left(-\dfrac{1}{2}\boldsymbol{\vartheta}_{t-i}^{l}\left(n,k\right)\right)\right|\nonumber \\
 & \hspace{-2pt}\le\hspace{-2pt}\max_{l\in\mathbb{N}_{L}^{+}}\left|\theta_{t-i}\hspace{-2pt}\left(\boldsymbol{x}_{j}\right)\hspace{-2pt}-\hspace{-2pt}\boldsymbol{\vartheta}_{t-i}^{l}\hspace{-2pt}\left(n,k\right)\right|\hspace{-2pt}\le\hspace{-2pt}\max_{l\in\mathbb{N}_{L}^{+}}\gamma\hspace{-2pt}\left\Vert \boldsymbol{\theta}_{t-i}\hspace{-2pt}\left(\boldsymbol{x}_{j}\right)\right\Vert _{2}\hspace{-2pt}n\rho^{n}\nonumber \\
 & \hspace{-2pt}\le\hspace{-2pt}\gamma\delta S^{1.5}CN\left(1+\log\left(T+1\right)\right)n\rho^{n}.
\end{flalign}
Using the above into \eqref{eq:Theorem_2}, we arrive at the inequality
\begin{flalign}
\left\Vert \boldsymbol{E}_{t}-\widetilde{\boldsymbol{E}}_{t}^{k}\left(n\right)\right\Vert _{1} & \hspace{-2pt}\le\hspace{-2pt}\dfrac{\gamma\delta S^{1.5}CN\left(1+\log\left(T+1\right)\right)n\rho^{n}}{\lambda_{inf}^{Nt/4}}\sum_{i\in\mathbb{N}_{t}}\dfrac{1}{\lambda_{inf}^{Ni/4}}\nonumber \\
 & \hspace{-2pt}\le\hspace{-2pt}\dfrac{\gamma\delta S^{1.5}CN\left(1+\log\left(T+1\right)\right)n\rho^{n}}{\lambda_{inf}^{Nt/4}}\left(t+\dfrac{1}{\lambda_{inf}^{N/4}}\right)\nonumber \\
 & \hspace{-2pt}\le\hspace{-2pt}\gamma\delta S^{1.5}C\left(1+\log\left(T+1\right)\right)n\rho^{n}\left(\hspace{-2pt}\dfrac{Nt}{\lambda_{inf}^{Nt/4}}\hspace{-2pt}+\hspace{-2pt}\dfrac{N\left(t+1\right)}{\lambda_{inf}^{N\left(t+1\right)/4}}\hspace{-2pt}\right)\nonumber \\
 & \hspace{-2pt}\le\hspace{-2pt}\dfrac{8\gamma\delta S^{1.5}C\left(1+\log\left(T+1\right)\right)n\rho^{n}}{\log\left(\lambda_{inf}\right)\lambda_{inf}^{1/\log\left(\lambda_{inf}\right)}}.
\end{flalign}
But because $\lambda_{inf}^{1/\log\left(\lambda_{inf}\right)}\hspace{-2pt}\equiv e$
and $\left(\log\left(\lambda_{inf}\right)\right)^{-1}\hspace{-2pt}\le\hspace{-2pt}1$,
we get
\begin{align}
\left\Vert \boldsymbol{E}_{t}\hspace{-2pt}-\hspace{-2pt}\widetilde{\boldsymbol{E}}_{t}^{k}\left(n\right)\right\Vert _{1}\hspace{-2pt} & \le3\gamma\delta S^{1.5}C\left(1\hspace{-2pt}+\hspace{-2pt}\log\left(T+1\right)\right)n\rho^{n},
\end{align}
holding true for all $t\in\mathbb{N}_{T}$, with probability at least
$1-\left(T+1\right)^{1-CN}\exp\left(-CN\right)$, provided that
\begin{align}
n-\tau\log\left(\gamma n\right) & \hspace{-2pt}\ge\hspace{-2pt}\tau\log\left(\hspace{-2pt}\dfrac{2\delta S^{1.5}C\left(1+\log\hspace{-2pt}\left(T+1\right)\right)}{\log\left(\lambda_{inf}\right)}\hspace{-2pt}\right)\hspace{-2pt}.\label{eq:Theorem_3}
\end{align}
Finally, choose an $\varepsilon\in\left(0,1\right)$. Then, a sufficient
condition such that $\left\Vert \boldsymbol{E}_{t}\hspace{-2pt}-\hspace{-2pt}\widetilde{\boldsymbol{E}}_{t}^{k}\left(n\right)\right\Vert _{1}\hspace{-2pt}\le\hspace{-2pt}\varepsilon$,
for all $t\in\mathbb{N}_{T}$ and $k\in\mathbb{N}_{S}^{+}$, is 
\begin{equation}
3\gamma\delta S^{1.5}C\left(1+\log\left(T+1\right)\right)n\rho^{n}\le\varepsilon,
\end{equation}
or, after some algebra,
\begin{align}
n-\tau\log\left(\gamma n\right) & \ge\tau\log\left(\hspace{-2pt}\dfrac{3\delta S^{1.5}C\left(1+\log\hspace{-2pt}\left(T+1\right)\right)}{\varepsilon}\hspace{-2pt}\right)\hspace{-2pt}.\label{eq:Theorem_4}
\end{align}
Of course, the bounds \eqref{eq:Theorem_3} and \eqref{eq:Theorem_4}
must hold at the same time. Thus, for each $t\in\mathbb{N}_{T}$ ,
$\varepsilon\in\left(0,1\right)$ and choosing
\begin{equation}
\eta\triangleq\delta\max\left\{ 2/\log\left(\lambda_{inf}\right),3\right\} >1,
\end{equation}
it will be true that $\sup_{t\in\mathbb{N}_{T}}\sup_{k\in\mathbb{N}_{S}^{+}}\left\Vert \boldsymbol{E}_{t}-\widetilde{\boldsymbol{E}}_{t}^{k}\left(n\right)\right\Vert _{1}\le\varepsilon$,
with probability at least $1-\left(T+1\right)^{1-CN}$\linebreak{}
$\exp\left(-CN\right)$, provided that \eqref{eq:CONDITION} holds,
completing the proof.
\end{proof}

Our main result follows, establishing a fundamental lower bound on
the minimal number of consensus steps, such that the distributed filter
remains \textit{uniformly} stable with a prescribed accuracy level,
within finite operational horizon and across all sensors in the network.
\begin{thm}
\textbf{\textup{(Stability of the Distributed HMM Estimator)}} \label{Theorem_Filtering_ADMM-2}
Fix a natural $1\le T<\infty$ and choose $\varepsilon\in\left(0,1\right]$,
$m\ge0$, $C\ge1$, and a number of sensors $S\ge2$. Then, there
exists a constant $c>0$ such that, as long as $n\in\mathbb{N}^{\left\lfloor 2\epsilon_{max}+2\right\rfloor }$
and
\begin{flalign}
n-\tau\log\left(\gamma n\right) & \ge c\tau CNT\log\hspace{-2pt}\left(\dfrac{CST}{\varepsilon}\right)+\tau m\nonumber \\
 & \equiv{\cal O}\left(\hspace{-2pt}\tau CNT\log\hspace{-2pt}\left(\dfrac{CST}{\varepsilon}\right)\hspace{-2pt}\hspace{-2pt}\right)+\tau m,\label{eq:CONDITION-2}
\end{flalign}
the absolute error between $\pi_{t\left|\mathscr{Y}_{t}\right.}$
and $\widetilde{\pi}_{t\left|\mathscr{Y}_{t}\right.}^{k}\left(n\right)$
satisfies
\begin{equation}
\sup_{k\in\mathbb{N}_{S}^{+}}\sup_{t\in\mathbb{N}_{T}}\left\Vert \pi_{t\left|\mathscr{Y}_{t}\right.}-\widetilde{\pi}_{t\left|\mathscr{Y}_{t}\right.}^{k}\left(n\right)\right\Vert _{1}\le\varepsilon\exp\left(-m\right),
\end{equation}
with probability at least $1-\left(T+1\right)^{1-CN}\exp\left(-CN\right)$.
That is, provided \eqref{eq:CONDITION-2} holds locally at each sensor
$k\in\mathbb{N}_{S}^{+}$, $\pi_{t\left|\mathscr{Y}_{t}\right.}$
equals $\widetilde{\pi}_{t\left|\mathscr{Y}_{t}\right.}^{k}\left(n\right)$
\textbf{within }a \textbf{global, exponentially decreasing} factor
$\varepsilon\exp\left(-m\right)$, with overwhelmingly high probability.
\end{thm}
\begin{proof}[Proof of Theorem \ref{Theorem_Filtering_ADMM-2}]
First, by construction and invoking the reverse triangle inequality,
it is easy to show that the error in the $\ell_{1}$-norm between
$\pi_{t\left|\mathscr{Y}_{t}\right.}$ and $\widetilde{\pi}_{t\left|\mathscr{Y}_{t}\right.}^{k}\left(n\right)$
may be upper bounded as
\begin{equation}
\left\Vert \pi_{t\left|\mathscr{Y}_{t}\right.}-\widetilde{\pi}_{t\left|\mathscr{Y}_{t}\right.}^{k}\left(n\right)\right\Vert _{1}\le2\dfrac{\left\Vert \boldsymbol{E}_{t}-\widetilde{\boldsymbol{E}}_{t}^{k}\left(n\right)\right\Vert _{1}}{\left\Vert \boldsymbol{E}_{t}\right\Vert _{1}},
\end{equation}
for all $k\in\mathbb{N}_{S}^{+}$, $t\in\mathbb{N}_{T}$ and any qualifying
$n\in\mathbb{N}^{2}$. In light of Theorem \ref{Theorem_Filtering_ADMM},
as long as the respective conditions on $n$ are fulfilled, it will
be true that, for any $\varepsilon\in\left(0,1\right)$,
\begin{equation}
\left\Vert \pi_{t\left|\mathscr{Y}_{t}\right.}-\widetilde{\pi}_{t\left|\mathscr{Y}_{t}\right.}^{k}\left(n\right)\right\Vert _{1}\left(\omega\right)\le\dfrac{2\varepsilon}{\left\Vert \boldsymbol{E}_{t}\right\Vert _{1}\left(\omega\right)},
\end{equation}
for all $\omega\in{\cal T}_{T}$, with ${\cal P}\left({\cal T}_{T}\right)\ge1-\left(T+1\right)^{1-CN}\exp\left(-CN\right)$.
Now, as far as $\left\Vert \boldsymbol{E}_{t}\right\Vert _{1}$ is
concerned, it has been shown by the authors that \cite{KalPetGRID2014,KalPetNonlinear_2015}
\begin{align}
\left\Vert \boldsymbol{E}_{t}\right\Vert _{1}\left(\omega\right) & \ge\lambda_{sup}^{-N\left(T+1\right)/2}\exp\hspace{-2pt}\left(\hspace{-2pt}\hspace{-2pt}-\dfrac{\left(\hspace{-2pt}\sqrt{\beta CN\hspace{-2pt}\left(1\hspace{-2pt}+\hspace{-2pt}\log\hspace{-2pt}\left(T\hspace{-2pt}+\hspace{-2pt}1\right)\right)}\hspace{-2pt}+\hspace{-2pt}\mu_{sup}\right)^{\hspace{-2pt}2}\hspace{-2pt}\left(T\hspace{-2pt}+\hspace{-2pt}1\right)}{2\lambda_{inf}}\hspace{-2pt}\right)\hspace{-2pt}\hspace{-2pt}\triangleq\hspace{-2pt}B,
\end{align}
for all $\omega$ in the same measurable set ${\cal T}_{T}$, where
all our probabilistic statements made so far have taken place in.
As a result, and given that $\max\left\{ \lambda_{sup},e\right\} \equiv\lambda_{sup}$
by assumption, there exists a constant $c_{0}>0$, such that
\begin{equation}
\dfrac{1}{B}\lambda_{sup}^{-c_{0}CNT\log\left(T\right)}\le1.
\end{equation}
In order to show this, note that
\begin{alignat}{1}
\dfrac{\left(T\hspace{-2pt}+\hspace{-2pt}1\right)}{2}\hspace{-2pt}\left(\hspace{-2pt}N\hspace{-2pt}+\hspace{-2pt}\dfrac{\left(\hspace{-2pt}\sqrt{\beta CN\hspace{-2pt}\left(1\hspace{-2pt}+\hspace{-2pt}\log\hspace{-2pt}\left(T\hspace{-2pt}+\hspace{-2pt}1\right)\right)}\hspace{-2pt}+\hspace{-2pt}\mu_{sup}\right)^{2}}{\lambda_{inf}}\right)\le & \dfrac{\left(T\hspace{-2pt}+\hspace{-2pt}1\right)}{2}\hspace{-2pt}\left(\hspace{-2pt}N\hspace{-2pt}+\hspace{-2pt}\dfrac{\beta CN\hspace{-2pt}\left(1\hspace{-2pt}+\hspace{-2pt}\log\hspace{-2pt}\left(T\hspace{-2pt}+\hspace{-2pt}1\right)\right)\left(\mu_{sup}\hspace{-2pt}+\hspace{-2pt}1\right)^{2}}{\lambda_{inf}}\right)\nonumber \\
\le & \left(T\hspace{-2pt}+\hspace{-2pt}1\right)CN\hspace{-2pt}\left(1\hspace{-2pt}+\hspace{-2pt}\log\hspace{-2pt}\left(T\hspace{-2pt}+\hspace{-2pt}1\right)\right)\hspace{-2pt}\left(\hspace{-2pt}\dfrac{1}{2}\hspace{-2pt}+\hspace{-2pt}\dfrac{\beta\left(\mu_{sup}\hspace{-2pt}+\hspace{-2pt}1\right)^{2}}{2\lambda_{inf}}\right)\nonumber \\
\triangleq & \,c_{1}\left(T\hspace{-2pt}+\hspace{-2pt}1\right)CN\hspace{-2pt}\left(1\hspace{-2pt}+\hspace{-2pt}\log\hspace{-2pt}\left(T\hspace{-2pt}+\hspace{-2pt}1\right)\right)\nonumber \\
\le & \,c_{1}c_{2}TCN\hspace{-2pt}\log\left(T\right)\triangleq c_{0}TCN\hspace{-2pt}\log\left(T\right),
\end{alignat}
whenever $T\ge1$, for some positive constants $c_{1}$ and $c_{2}$.
Consequently, if, for each $t\in\mathbb{N}_{T}$, one chooses
\begin{equation}
\varepsilon\hspace{-2pt}\equiv\hspace{-2pt}\dfrac{1}{2}\lambda_{sup}^{-c_{0}CNT\hspace{-2pt}\log\left(T\right)}\widetilde{\varepsilon}\exp\hspace{-2pt}\left(-m\right)\hspace{-2pt}\le\hspace{-2pt}1,\label{eq:FINAL_EPSILON}
\end{equation}
for some $\widetilde{\varepsilon}\in\left(0,1\right]$ and $m\ge0$,
it will be true that
\begin{equation}
\sup_{k\in\mathbb{N}_{S}^{+}}\sup_{t\in\mathbb{N}_{T}}\left\Vert \pi_{t\left|\mathscr{Y}_{t}\right.}-\widetilde{\pi}_{t\left|\mathscr{Y}_{t}\right.}^{k}\left(n\right)\right\Vert _{1}\le\widetilde{\varepsilon}\exp\left(-m\right),
\end{equation}
provided that
\begin{equation}
n-\tau\log\left(\gamma n\right)\ge\tau\log\left(\dfrac{\eta S^{1.5}C\left(1+\log\left(T+1\right)\right)}{\varepsilon}\right).
\end{equation}
Replacing $\varepsilon$ with the expression in \eqref{eq:FINAL_EPSILON}
and for $S\ge2$, the logarithmic term above may be upper bounded
as
\begin{flalign}
\log\hspace{-2pt}\left(\hspace{-2pt}\dfrac{\eta S^{1.5}C\left(1+\log\left(T+1\right)\right)}{\varepsilon}\hspace{-2pt}\right)\equiv & \log\hspace{-2pt}\left(\hspace{-2pt}\dfrac{2\eta S^{1.5}C\left(1+\log\left(T+1\right)\right)}{\widetilde{\varepsilon}}\lambda_{sup}^{c_{0}CNT\log\left(T\right)}\hspace{-2pt}\exp\left(m\right)\hspace{-2pt}\right)\nonumber \\
\le & \,m\hspace{-2pt}+\hspace{-2pt}c_{0}\hspace{-2pt}\log\hspace{-2pt}\left(\lambda_{sup}\right)\hspace{-2pt}CNT\log\hspace{-2pt}\left(T\right)\hspace{-2pt}+\hspace{-2pt}\log\hspace{-2pt}\left(\hspace{-2pt}\dfrac{2c_{2}\eta S^{1.5}C\hspace{-2pt}\log\left(T\right)}{\widetilde{\varepsilon}}\hspace{-2pt}\right)\nonumber \\
\triangleq & \,m\hspace{-2pt}+\widetilde{c}_{0}CNT\log\hspace{-2pt}\left(T\right)\hspace{-2pt}+\hspace{-2pt}\log\hspace{-2pt}\left(\hspace{-2pt}\dfrac{2c_{2}\eta S^{1.5}C\hspace{-2pt}\log\left(T\right)}{\widetilde{\varepsilon}}\hspace{-2pt}\right)\nonumber \\
\le & \,m+\widetilde{c}_{0}CNT\log\hspace{-2pt}\left(T\right)\hspace{-2pt}+\widetilde{c}_{2}\log\hspace{-2pt}\left(\dfrac{CST}{\widetilde{\varepsilon}}\right)\nonumber \\
\le & \,m+\left(\widetilde{c}_{0}+\widetilde{c}_{2}\right)CNT\log\hspace{-2pt}\left(\dfrac{CST}{\widetilde{\varepsilon}}\right)\nonumber \\
\triangleq & \,m+cCNT\log\hspace{-2pt}\left(\dfrac{CST}{\widetilde{\varepsilon}}\right),
\end{flalign}
for some other positive constant $\widetilde{c}_{2}$. Putting it
altogether and renaming $\widetilde{\varepsilon}$ to $\varepsilon$,
we have finally shown that, for any $\varepsilon\in\left(0,1\right]$,
$m\ge0$, $S\ge2$ and any natural $1\le T<\infty$, there exists
a constant $c>0$, such that whenever
\begin{flalign}
n-\tau\log\left(\gamma n\right) & \ge c\tau CNT\hspace{-2pt}\log\hspace{-2pt}\left(\dfrac{CST}{\varepsilon}\right)+\tau m,
\end{flalign}
it is true that
\begin{equation}
\sup_{k\in\mathbb{N}_{S}^{+}}\sup_{t\in\mathbb{N}_{T}}\left\Vert \pi_{t\left|\mathscr{Y}_{t}\right.}-\widetilde{\pi}_{t\left|\mathscr{Y}_{t}\right.}^{k}\left(n\right)\right\Vert _{1}\le\varepsilon\exp\left(-m\right),
\end{equation}
with probability at least $1-\left(T+1\right)^{1-CN}\exp\left(-CN\right)$,
and the proof is complete.
\end{proof}
\begin{rem}
The linear-plus-logarithmic form of the LHS of \eqref{eq:CONDITION-2}
does not seem to have a specific physical interpretation, related
to the conclusions of Theorem \ref{Theorem_Filtering_ADMM-2}. In
particular, the logarithmic term appearing on the LHS of \eqref{eq:CONDITION-2}
results as an artifact of our analytical development. Nevertheless,
it is true that, for relatively small values of $\tau$ and $\gamma$
and for sufficiently large $n$, the logarithmic term $\tau\log\left(\gamma n\right)$
is insignificant, compared to $n$, and may be considered relatively
constant, that is, $n-\tau\log\left(\gamma n\right)\approx n-\tau\kappa\left(\gamma\right)$,
for some $\kappa$, possibly dependent on $\gamma$, but independent
of $n$. Therefore, the logarithmic artifact on the LHS of \eqref{eq:CONDITION-2}
may be heuristically considered as a positive, constant bias on the
RHS as $n$ increases, which, however, does not affect the asymptotic
behavior of the result.\hfill{}\ensuremath{\blacksquare}
\end{rem}

\section{Discussion \& Some Numerical Simulations}

In this Section, we present some numerical simulations, experimentally
validating some of the properties of the distributed filtering scheme
considered in this paper, as well as a relevant discussion.

For our numerical simulations, we assume that we are given a DNA with
$S\hspace{-2pt}\equiv\hspace{-2pt}60$ sensors. The connectivity threshold
of the underlying RGG $r$ is set to $0.2$, representing a relatively
sparsely connected network. Each sensor in the network observes noisy
versions of a Markov chain with $L\hspace{-2pt}\equiv\hspace{-2pt}4$
states, state space ${\cal X}\hspace{-2pt}\equiv\hspace{-2pt}\left\{ 0.7,1/2,1,1/3\right\} $,
a stochastic transition matrix
\begin{equation}
\boldsymbol{P}\equiv\begin{bmatrix}0.4 & 0.25 & 0.2 & 0.3\\
0.3 & 0.25 & 0.3 & 0.3\\
0.2 & 0.25 & 0.3 & 0.2\\
0.1 & 0.25 & 0.2 & 0.2
\end{bmatrix}\in\mathbb{R}^{4\times4},
\end{equation}
and initial distribution $\pi_{-1}\equiv\left[1\,0\,0\,0\right]^{\boldsymbol{T}}$.
Regarding the observation model of the HMM under consideration, we
assume that $N_{k}\equiv2$, $\boldsymbol{\mu}_{t}^{k}\left(\boldsymbol{x}\right)\equiv\boldsymbol{\mu}\left(\boldsymbol{x}\right)\equiv-\left[\sin\left(\boldsymbol{x}\right)\,\sin\left(\boldsymbol{x}\right)\right]^{\boldsymbol{T}}$
and $\boldsymbol{\Sigma}_{t}^{k}\left(\boldsymbol{x}\right)\equiv\boldsymbol{\Sigma}^{k}$,
with $\boldsymbol{\Sigma}^{k}\left(i,j\right)\equiv2\exp\left(-2\left|i-j\right|\right)$,
for all $\left(i,j\right)\in\mathbb{N}_{2}^{+}\times\mathbb{N}_{2}^{+}$
and for all $k\in\mathbb{N}_{S}^{+}$. As far as the distributed averaging
scheme is concerned, ${\bf S}$ is suboptimally chosen according to
the maximum degree chain construction and $\epsilon$ is optimally
chosen, according to what was previously stated in Section \ref{sec:Preliminaries}.

As a potential practical example for the model setting considered
above, sensors could be deployed in an industrial facility, monitoring
the hidden state (a mode) of a composite chemical process, at various
locations inside the facility, at certain time intervals. In order
to increase the information diversity at each location, multiple simultaneous
measurements may be recorded, resulting in a correlated observation
model, such as the one considered in our simulations. The distributed
filtering scheme considered in this paper could then be employed,
in order to perform global state estimation without the need of a
dedicated fusion center.
\begin{figure*}
\centering$\hspace{-10pt}$\subfloat[\label{fig:Consensus_800}]{\centering\includegraphics[clip,scale=0.61]{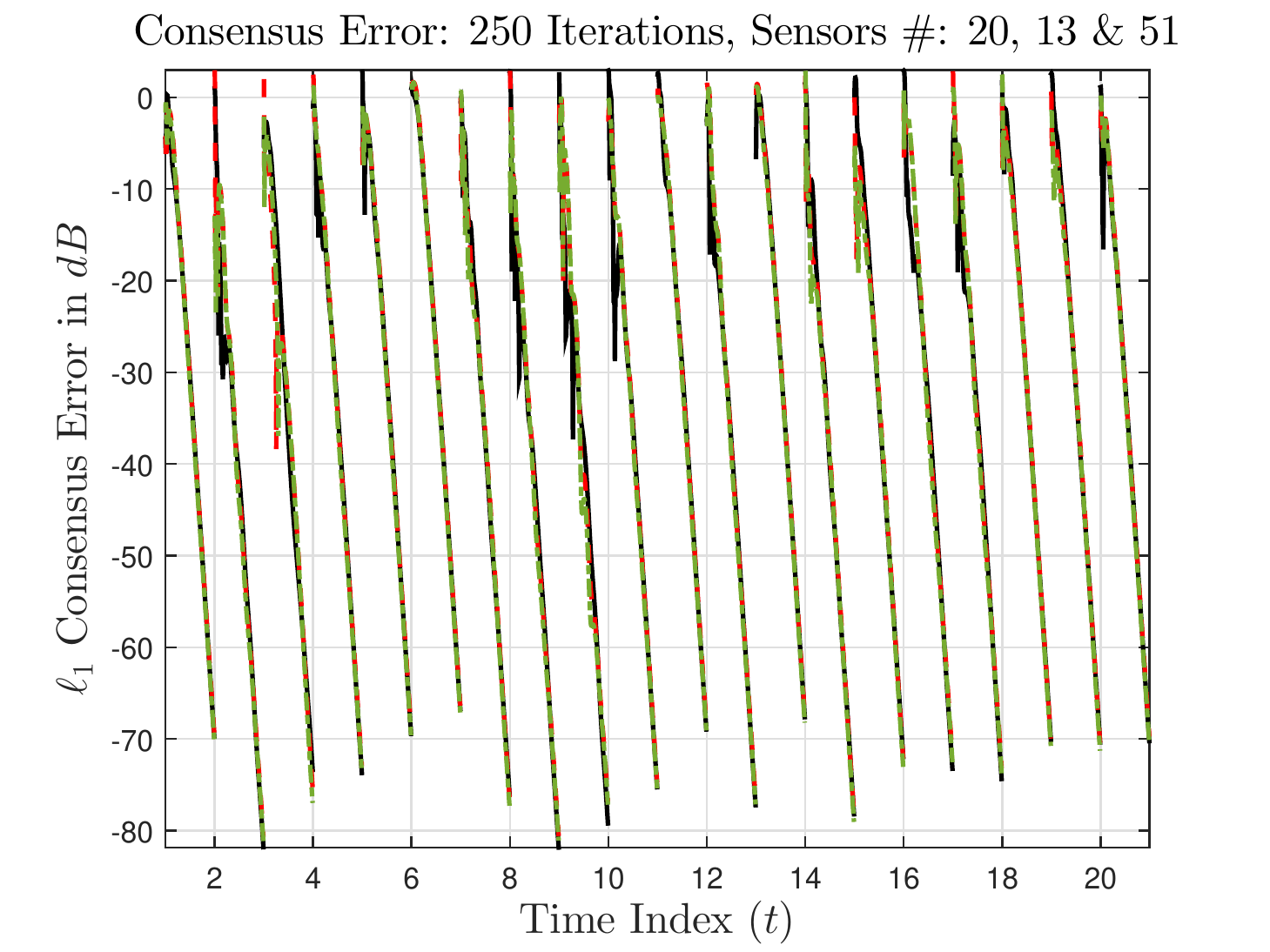}

}$\hspace{-16pt}$\subfloat[\label{fig:MMSE_800}]{\centering\includegraphics[scale=0.61]{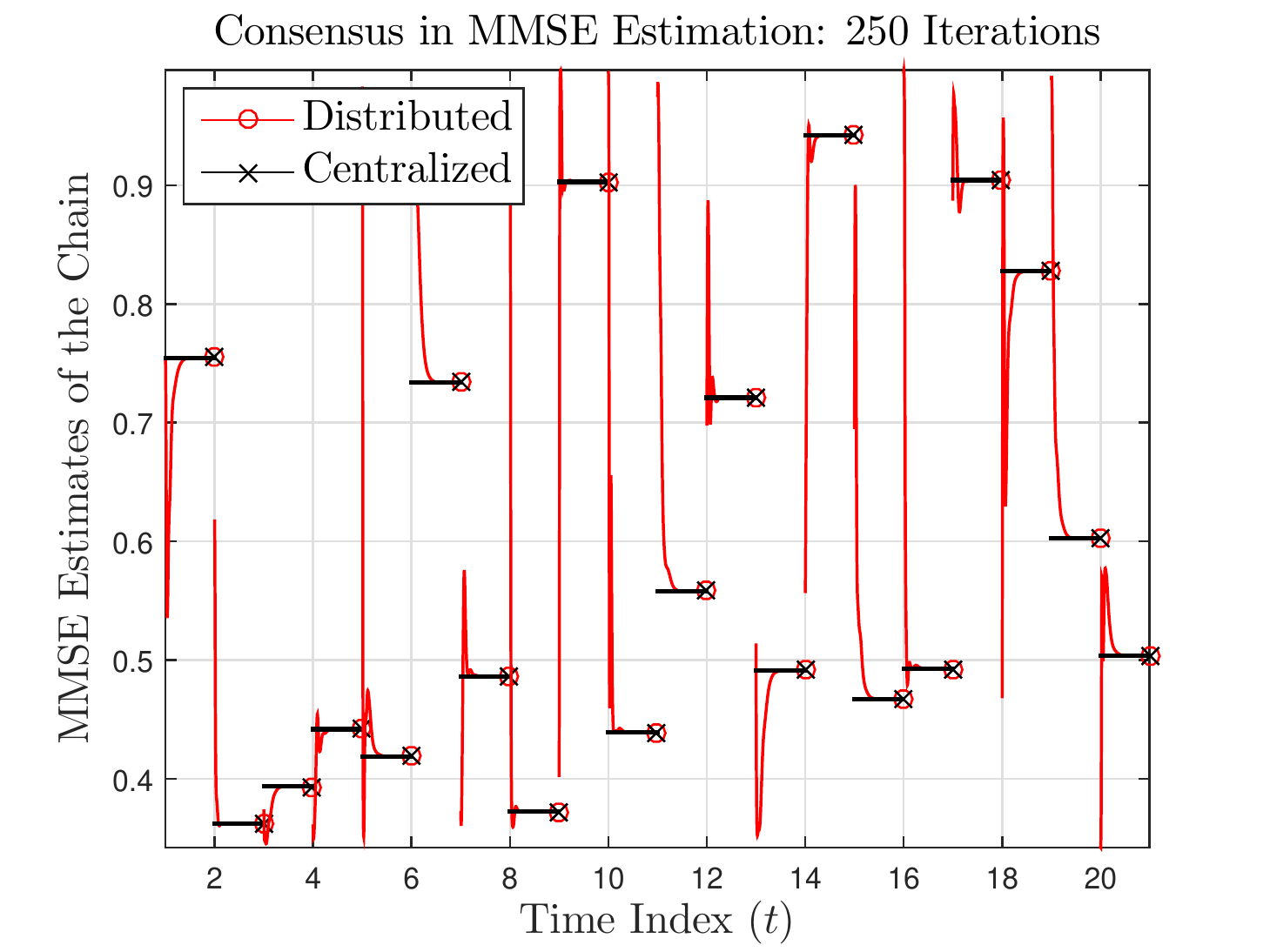}}\caption{Case $2$: $250$ consensus iterations for $T\equiv20$. (a) $\ell_{1}$
consensus error curves. (b) Consensus in MMSE estimation. Between
each pair of times, we plot the evolution of the respective distributed
estimate.}
\end{figure*}

As we discussed in Section \ref{sec:MAIN_RESULTS}, Theorem \ref{Theorem_Filtering_ADMM-2}
provides a \textit{sufficient} condition on the number of consensus
iterations, such that the $\ell_{1}$ consensus error between the
centralized and distributed versions of the posterior measure $\pi_{t\left|\mathscr{Y}_{t}\right.}$
is \textit{exponentially small}, \textit{uniformly} through the whole
operational horizon of interest and across sensors, at the same time.

Fig. \ref{fig:Consensus_800} shows the $\ell_{1}$ consensus error
for $T\equiv20$, where the number of iterations is set to $250$.
Results are shown for three randomly chosen sensors in the network.
It is important to note that the number of iterations has been chosen
such that, between each pair of consecutive times, the consensus error
\textit{does not reach machine precision}. This is critical in order
to assess the uniform properties of the consensus error, through time
and across sensors, and in order to verify the claims of our main
result, Theorem \ref{Theorem_Filtering_ADMM-2}. Indeed, as depicted
in Fig. \ref{fig:Consensus_800}, one can readily identify a perfectly
acceptable uniform upper bound for the consensus error, for all $20$
times of interest, say $-60\,dB$. Additionally and in perfect agreement
with Theorem \ref{Theorem_Filtering_ADMM-2}, for the chosen number
of consensus iterations, convergence of the averaging scheme is indeed
\textit{eventually} exponential, meaning that after some specific
number of iterations, such type of convergence is indeed achievable.
This is due to the fact that the consensus error curves appear roughly
as straight lines in the logarithmic domain, a fact which translates
to an exponential decay in the linear domain, respectively. The effect
of this behavior on estimation quality can also be observed in Fig.
\ref{fig:MMSE_800}, which shows the MMSE estimates of the chain obtained
by exploiting the respective distributed posterior estimates.

\textit{However}, it is important to note that Theorem \ref{Theorem_Filtering_ADMM-2}
does not provide a necessary condition on the behavior of the consensus
error. This means that, within the framework of this paper, stability
of distributed filtering is not theoretically predictable, in the
case where condition \eqref{eq:CONDITION-2} is not satisfied. In
fact, for almost all numerical simulations we have conducted, the
proposed scheme performed exceptionally well in terms of stability,
even for relatively smaller number of iterations, at least as far
as the given experimental setting is concerned. In particular, the
error due to the imperfect agreement on the version of the posterior
at each sensor does not significantly accumulate, as time progresses.

It is worth mentioning that, although loglinear dependence on $N\ge S$
and $T$ of the required number of consensus iterations is not really
evident from our numerical experiments, this is apparently due to
the relatively small number of sensors and the relatively ``easy''
nonlinear filtering problem considered. The aforementioned \textit{asymptotic}
loglinear behavior would certainly be revealed when the number of
sensors and/or complexity of the HMM under consideration increases.

Finally, we would like to note that the method advocated in this paper
is expected to perform better than the one in \cite{DHMM1_2010},
since it utilizes more consensus iterations (fast time) per filtering
time instant (slow time); \cite{DHMM1_2010} employs a single information
exchange among the sensors, per filtering time instant. In \cite{DHMM1_2010},
the filtering estimates at each sensor do not, in general, coincide
within reasonable bounds, that is, consensus is not always achieved.
In fact, consensus will not be achieved unless the network is \textit{strongly
connected}. On the other hand, in the distributed filtering scheme
analyzed in this work, the information collected at each sensor is
diffused uniformly through the whole network. As a result, consensus
is achieved and any potential heterogeneity among the observations
at the sensors is efficiently exploited.

\section{Conclusion}

We have studied stability of distributed inference in Gaussian HMMs
with finite state space, exploiting ADMM-like based distributed average
consensus. We have considered a distributed filtering scheme, which
relies on the distributed evaluation of the likelihood part of the
centralized nonlinear estimator under consideration. Basically assuming
the same number of consensus iterations between any two consecutive
sensor observations, we have provided a complete characterization
of a minimal number of consensus steps, guaranteeing uniform stability
of the resulting distributed nonlinear filter, within a finite operational
horizon and across all sensors in the DNA under consideration. We
have presented a fundamental result, showing that $\varepsilon$-stability
of the distributed filtering process depends only loglinearly on the
horizon of interest, $T$, and the dimensionality of the observation
process, $N$, and logarithmically on $1/\varepsilon$. Moreover,
strictly fulfilling this total loglinear bound incurs a fully quantified
exponential decay in the consensus error. Our bounds are universal,
in the sense that they are independent of the structure of the Gaussian
HMM under consideration.

\section*{Appendix A: Proof of Theorem \ref{Spectrum_of_M}}

As implied by the statement of Theorem \ref{Spectrum_of_M}, we start
by studying the dependence of the spectrum of ${\bf M}$ on that of
${\bf S}$ and on $\epsilon>0$. Of course, in order to determine
the spectrum of ${\bf M}$, we look for the $2S$ solutions of the
equation $\det\left({\bf M}-\lambda{\bf I}\right)\equiv0$, where
\begin{align}
{\bf M}-\lambda{\bf I} & \equiv\hspace{-2pt}\begin{bmatrix}\dfrac{\epsilon}{1+\epsilon}{\bf S}+\left(1-\lambda\right){\bf I} & -\dfrac{\epsilon}{2\left(1+\epsilon\right)}\hspace{-2pt}\left({\bf S}+{\bf I}\right)\\
{\bf I} & -\lambda{\bf I}
\end{bmatrix}\hspace{-2pt}.\label{eq:EIGEN_Matrix}
\end{align}
Since the submatrices ${\bf I}$ and $-\lambda{\bf I}$ obviously
commute in \eqref{eq:EIGEN_Matrix}, we may invoke (\cite{Silvester2000DET},
Theorem 3) and write
\begin{flalign}
\det\left({\bf M}\hspace{-2pt}-\hspace{-2pt}\lambda{\bf I}\right) & \hspace{-2pt}\equiv\hspace{-2pt}\det\hspace{-2pt}\left(-\dfrac{\epsilon\lambda}{1+\epsilon}{\bf S}-\lambda\left(1-\lambda\right)\hspace{-2pt}{\bf I}+\dfrac{\epsilon}{2\hspace{-2pt}\left(1+\epsilon\right)}\hspace{-2pt}\left({\bf S}+{\bf I}\right)\hspace{-2pt}\right)\nonumber \\
 & \hspace{-2pt}\equiv\hspace{-2pt}\dfrac{\left(-1\right)^{S}}{\left(2\left(1+\epsilon\right)\right)^{S}}\hspace{-2pt}\det\hspace{-2pt}\left(\epsilon\left(2\lambda\hspace{-2pt}-\hspace{-2pt}1\right){\bf S}\hspace{-2pt}+\hspace{-2pt}\left(2\lambda\left(1\hspace{-2pt}+\hspace{-2pt}\epsilon\right)\hspace{-2pt}\left(1\hspace{-2pt}-\hspace{-2pt}\lambda\right)\hspace{-2pt}-\hspace{-2pt}\epsilon\right){\bf I}\right)\hspace{-2pt}.\hspace{-2pt}\hspace{-2pt}\label{eq:Pre_Equation}
\end{flalign}
Using \eqref{eq:Pre_Equation}, after some algebra and via a diagonalization
of the symmetric ${\bf S}$, the equation $\det\hspace{-2pt}\left({\bf M}\hspace{-2pt}-\hspace{-2pt}\lambda{\bf I}\right)\hspace{-2pt}\equiv\hspace{-2pt}0$
can be easily shown to be equivalent to $S$ quadratic equations of
the form
\begin{equation}
-2\left(1+\epsilon\right)\lambda^{2}+2\left(1+\epsilon+\epsilon\lambda_{{\bf S}}\right)\lambda-\epsilon\left(1+\lambda_{{\bf S}}\right)\equiv0,\label{eq:Q_Equation}
\end{equation}
where, for notational brevity, $\lambda_{{\bf S}}$ hereafter denotes
any of the $S$ eigenvalues comprising the spectrum of ${\bf S}$.
Solving \eqref{eq:Q_Equation} results in the functional relation
\begin{equation}
\lambda_{\boldsymbol{M}}^{\pm}\hspace{-2pt}\left(\lambda_{{\bf S}},\epsilon\right)\hspace{-2pt}\equiv\hspace{-2pt}\dfrac{1+\epsilon+\epsilon\lambda_{{\bf S}}\pm\hspace{-2pt}\sqrt{1\hspace{-2pt}+\hspace{-2pt}\epsilon^{2}\left(\left(\lambda_{{\bf S}}\right)^{2}\hspace{-2pt}-\hspace{-2pt}1\right)}}{2\left(1+\epsilon\right)},\label{eq:M_Spectrum}
\end{equation}
for all $\lambda_{{\bf S}}\in\left(-1,1\right]$ and $\epsilon>0$,
corresponding to the respective expression shown in the statement
of Theorem \ref{Spectrum_of_M}.

The rest of the proof deals with characterizing \eqref{eq:M_Spectrum}
as a function of $\lambda_{{\bf S}}$ and $\epsilon$, in order to
determine the SLEM of ${\bf M}$, $\rho\equiv\rho\left(\epsilon,\lambda_{{\bf S}}^{1},\ldots,\lambda_{{\bf S}}^{S}\right)$
and, further, studying the problem of minimizing that SLEM with respect
to $\epsilon$. Fix $\epsilon>0$. When $\lambda_{{\bf S}}\equiv1$,
then $\lambda_{\boldsymbol{M}}^{+}\hspace{-2pt}\left(1,\epsilon\right)\equiv1\equiv\lambda_{\boldsymbol{M}}^{1}$
and $\lambda_{\boldsymbol{M}}^{-}\hspace{-2pt}\left(1,\epsilon\right)\equiv\epsilon/\left(1+\epsilon\right)$.
This implies that the SLEM we are looking for is either $\epsilon/\left(1+\epsilon\right)$,
or is given by some $\left|\lambda_{\boldsymbol{M}}^{\pm}\hspace{-2pt}\left(\lambda_{{\bf S}},\epsilon\right)\right|$,
for some $\lambda_{{\bf S}}\neq1$. Thus, the SLEM of ${\bf M}$ may
be expressed as
\begin{flalign}
\rho\hspace{-2pt}\left(\epsilon,\lambda_{{\bf S}}^{2},\ldots,\lambda_{{\bf S}}^{S}\right) & \hspace{-2pt}\equiv\hspace{-2pt}\max\hspace{-2pt}\left\{ \hspace{-2pt}\dfrac{\epsilon}{1+\epsilon},\left\{ \left|\lambda_{\boldsymbol{M}}^{j}\hspace{-2pt}\left(\lambda_{{\bf S}}^{i},\epsilon\right)\right|\right\} _{i\in\mathbb{N}_{S}^{2}}^{j\in\left\{ +,-\right\} }\hspace{-2pt}\right\} \nonumber \\
 & \hspace{-2pt}\equiv\hspace{-2pt}\max\hspace{-2pt}\left\{ \hspace{-2pt}\dfrac{\epsilon}{1+\epsilon},\max_{\lambda_{{\bf S}}\in\left\{ \lambda_{{\bf S}}^{i}\right\} _{i\in\mathbb{N}_{S}^{2}},j\in\left\{ +,-\right\} }\left|\lambda_{\boldsymbol{M}}^{j}\hspace{-2pt}\left(\lambda_{{\bf S}},\epsilon\right)\right|\hspace{-2pt}\right\} \hspace{-2pt}.\label{eq:max_max}
\end{flalign}
Let us focus on the inner maximization on the RHS of \eqref{eq:max_max}.
In the following, the cases where $\lambda_{{\bf S}}^{2}\ge0$ and
$\lambda_{{\bf S}}^{2}<0$ will be treated separately. As we will
see, the SLEM of ${\bf M}$ behaves quite differently under each of
the two aforementioned cases.

\textit{Suppose that} $\lambda_{{\bf S}}^{2}\ge0$. Due to the existence
of $\epsilon$ in the square root on the RHS of \eqref{eq:M_Spectrum},
it is reasonable to consider the following further two subcases:

$\bullet\;1+\epsilon^{2}\left(\left(\lambda_{{\bf S}}\right)^{2}\hspace{-2pt}-\hspace{-2pt}1\right)\hspace{-2pt}\ge\hspace{-2pt}0\,$:
If the condition on the left is true, then the corresponding eigenvalues
of ${\bf M}$, compactly expressed via \eqref{eq:M_Spectrum}, are
real. What is more, via an equivalence test, it can be easily shown
that $\lambda_{\boldsymbol{M}}^{\pm}\hspace{-2pt}\left(\lambda_{{\bf S}},\epsilon\right)>0$
(for both nonnegative and negative $\lambda_{{\bf S}}$, actually).
Consequently, in this case, $\left|\lambda_{\boldsymbol{M}}^{\pm}\hspace{-2pt}\left(\lambda_{{\bf S}},\epsilon\right)\right|\equiv\lambda_{\boldsymbol{M}}^{\pm}\hspace{-2pt}\left(\lambda_{{\bf S}},\epsilon\right)$,
for each fixed $\epsilon$. Under these considerations, because $\lambda_{\boldsymbol{M}}^{+}\hspace{-2pt}\left(\lambda_{{\bf S}},\epsilon\right)>\lambda_{\boldsymbol{M}}^{-}\hspace{-2pt}\left(\lambda_{{\bf S}},\epsilon\right)$
for any feasible $\lambda_{{\bf S}}$, as well as due to the easily
provable fact that, whenever $\lambda_{{\bf S}}\ge0$ and $\widetilde{\lambda}_{{\bf S}}<0$,
$\lambda_{\boldsymbol{M}}^{+}\hspace{-2pt}\left(\lambda_{{\bf S}},\epsilon\right)>\lambda_{\boldsymbol{M}}^{+}\hspace{-2pt}\left(\widetilde{\lambda}_{{\bf S}},\epsilon\right)$
(again via an equivalence test), it suffices to consider \textit{maximizing}
$\lambda_{\boldsymbol{M}}^{+}\hspace{-2pt}\left(\lambda_{{\bf S}},\epsilon\right)$
\textit{only for the case of a nonnegative} $\lambda_{{\bf S}}$.
Further, assuming that $\lambda_{{\bf S}}\ge0$ (always in addition
to the condition $1+\epsilon^{2}\left(\left(\lambda_{{\bf S}}\right)^{2}\hspace{-2pt}-\hspace{-2pt}1\right)\hspace{-2pt}\ge\hspace{-2pt}0$),
it is easy to show that $\partial\lambda_{\boldsymbol{M}}^{+}\hspace{-2pt}\left(\lambda_{{\bf S}},\epsilon\right)/\partial\lambda_{{\bf S}}>0$,
implying that $\lambda_{\boldsymbol{M}}^{+}\hspace{-2pt}\left(\lambda_{{\bf S}},\epsilon\right)$
is \textit{strictly increasing in} $\lambda_{{\bf S}}$, for any $\epsilon>0$.
Lastly, it is easy to show that $\lambda_{\boldsymbol{M}}^{+}\hspace{-2pt}\left(\lambda_{{\bf S}},\epsilon\right)>\epsilon/\left(1+\epsilon\right)$,
regardless of $\lambda_{{\bf S}}\ge0$ (a fact which will be useful
later).

$\bullet\;1+\epsilon^{2}\left(\left(\lambda_{{\bf S}}\right)^{2}\hspace{-2pt}-\hspace{-2pt}1\right)\hspace{-2pt}<\hspace{-2pt}0\,$:
In this case, the respective pairs of eigenvalues of ${\bf M}$ are
complex. Hence, it is true that
\begin{equation}
\left|\lambda_{\boldsymbol{M}}^{\pm}\hspace{-2pt}\left(\lambda_{{\bf S}},\epsilon\right)\right|\hspace{-2pt}\equiv\hspace{-2pt}\dfrac{\sqrt{\left(1+\epsilon+\epsilon\lambda_{{\bf S}}\right)^{2}-1\hspace{-2pt}-\hspace{-2pt}\epsilon^{2}\left(\left(\lambda_{{\bf S}}\right)^{2}\hspace{-2pt}-\hspace{-2pt}1\right)}}{2\left(1+\epsilon\right)}.
\end{equation}
By inspection, we have $\partial\left|\lambda_{\boldsymbol{M}}^{\pm}\hspace{-2pt}\left(\lambda_{{\bf S}},\epsilon\right)\right|/\partial\lambda_{{\bf S}}>0$.
This again implies that $\left|\lambda_{\boldsymbol{M}}^{\pm}\hspace{-2pt}\left(\lambda_{{\bf S}},\epsilon\right)\right|$
is \textit{strictly increasing in} $\lambda_{{\bf S}}$, regardless
of $\lambda_{{\bf S}}$, for any feasible choice of $\epsilon>0$.
Concerning comparison with $\epsilon/\left(1+\epsilon\right)$, looking
at $\left|\lambda_{\boldsymbol{M}}^{\pm}\hspace{-2pt}\left(\lambda_{{\bf S}},\epsilon\right)\right|$
as a function of $\epsilon$, it is true that
\begin{equation}
\left|\lambda_{\boldsymbol{M}}^{\pm}\hspace{-2pt}\left(\lambda_{{\bf S}},\epsilon\right)\right|\gtreqqless\dfrac{\epsilon}{1+\epsilon}\Leftrightarrow\epsilon\lesseqqgtr\dfrac{1+\lambda_{{\bf S}}}{1-\lambda_{{\bf S}}},\label{eq:Equivalence_1}
\end{equation}
which however holds true only for $\lambda_{{\bf S}}\ge0$, since
then and only then the inequality $\left(1+\lambda_{{\bf S}}\right)/\left(1-\lambda_{{\bf S}}\right)\ge1/\sqrt{1-\left(\lambda_{{\bf S}}\right)^{2}}$
is satisfied, where the RHS constitutes the lower bound for $\epsilon$,
due to the condition $1+\epsilon^{2}\left(\left(\lambda_{{\bf S}}\right)^{2}\hspace{-2pt}-\hspace{-2pt}1\right)\hspace{-2pt}<\hspace{-2pt}0$.

Now, for brevity, define the finite sets ${\cal C}\triangleq{\cal R}^{c}$,
with 
\begin{flalign}
{\cal R} & \triangleq\left\{ \lambda_{{\bf S}}\in\left\{ \lambda_{{\bf S}}^{i}\right\} _{i\in\mathbb{N}_{{\bf S}}^{2}}\left|1+\epsilon^{2}\left(\left(\lambda_{{\bf S}}\right)^{2}\hspace{-2pt}-\hspace{-2pt}1\right)\hspace{-2pt}\ge\hspace{-2pt}0\right.\hspace{-2pt}\right\} ,
\end{flalign}
both depending on the particular choice of $\epsilon>0$. Then, from
\eqref{eq:max_max} and combining the arguments made previously, the
SLEM under study may be expressed as
\begin{align}
\rho\left(\epsilon,\lambda_{{\bf S}}^{2},\ldots,\lambda_{{\bf S}}^{S}\right) & \hspace{-2pt}\equiv\hspace{-2pt}\max\hspace{-2pt}\left\{ \hspace{-2pt}\dfrac{\epsilon}{1\hspace{-2pt}+\hspace{-2pt}\epsilon},\max\hspace{-2pt}\left\{ \hspace{-2pt}\max_{0\le\lambda_{{\bf S}}\in{\cal R}}\lambda_{\boldsymbol{M}}^{+}\hspace{-2pt}\left(\lambda_{{\bf S}},\epsilon\right)\hspace{-2pt},\hspace{-2pt}\max_{\lambda_{{\bf S}}\in{\cal C}}\left|\lambda_{\boldsymbol{M}}^{+}\hspace{-2pt}\left(\lambda_{{\bf S}},\epsilon\right)\right|\right\} \hspace{-2pt}\right\} \hspace{-2pt}.\label{eq:max_max_max}
\end{align}
Apparently, there are two possibilities for $\lambda_{{\bf S}}^{2}$;
either $\lambda_{{\bf S}}^{2}\in{\cal C}$, or $\lambda_{{\bf S}}^{2}\in{\cal R}$.
If $\lambda_{{\bf S}}^{2}\in{\cal C}$ and due to the monotonicity
of the involved eigenvalues with respect to $\lambda_{{\bf S}}$,
it is true that $\max_{\lambda_{{\bf S}}\in{\cal C}}\left|\lambda_{\boldsymbol{M}}^{+}\hspace{-2pt}\left(\lambda_{{\bf S}},\epsilon\right)\right|\equiv\left|\lambda_{\boldsymbol{M}}^{+}\hspace{-2pt}\left(\lambda_{{\bf S}}^{2},\epsilon\right)\right|$.
Also, if there is any chance for the problem $\max_{0\le\lambda_{{\bf S}}\in{\cal R}}\lambda_{\boldsymbol{M}}^{+}\hspace{-2pt}\left(\lambda_{{\bf S}},\epsilon\right)$
to be feasible, its optimal value would attained for some $0\le\lambda_{{\bf S}}^{*}\in{\cal R}$,
with $\lambda_{{\bf S}}^{2}>\lambda_{{\bf S}}^{*}$. In this case,
since $\lambda_{{\bf S}}^{2}\in{\cal C}$ and by assumption $\lambda_{{\bf S}}^{2}\ge0$,
it readily follows that $\lambda_{{\bf S}}^{2}<\sqrt{\epsilon^{2}-1}/\epsilon$.
On the other hand, since $0\le\lambda_{{\bf S}}^{*}\in{\cal R}$ by
definition of the feasible set of the respective optimization problem,
it must be true that $\lambda_{{\bf S}}^{*}\ge\sqrt{\epsilon^{2}-1}/\epsilon$.
This, however, contradicts the fact that, also by definition, $\lambda_{{\bf S}}^{2}>\lambda_{{\bf S}}^{*}$.
Therefore, the problem $\max_{0\le\lambda_{{\bf S}}\in{\cal R}}\lambda_{\boldsymbol{M}}^{+}\hspace{-2pt}\left(\lambda_{{\bf S}},\epsilon\right)$
is infeasible, returning by convention the value $-\infty$. In the
case where $\lambda_{{\bf S}}^{2}\in{\cal R}$, it is true that $\max_{0\le\lambda_{{\bf S}}\in{\cal R}}\lambda_{\boldsymbol{M}}^{+}\hspace{-2pt}\left(\lambda_{{\bf S}},\epsilon\right)\equiv\lambda_{\boldsymbol{M}}^{+}\hspace{-2pt}\left(\lambda_{{\bf S}}^{2},\epsilon\right)$.
As above, if $\max_{\lambda_{{\bf S}}\in{\cal C}}\left|\lambda_{\boldsymbol{M}}^{+}\hspace{-2pt}\left(\lambda_{{\bf S}},\epsilon\right)\right|$
is feasible, its optimal value $\left|\lambda_{\boldsymbol{M}}^{+}\hspace{-2pt}\left(\lambda_{{\bf S}}^{*},\epsilon\right)\right|$
would be attained for some $\lambda_{{\bf S}}^{*}\in{\cal C}$, with
$\lambda_{{\bf S}}^{2}>\lambda_{{\bf S}}^{*}$. Via a simple equivalence
test, it can then be easily shown that $\lambda_{\boldsymbol{M}}^{+}\hspace{-2pt}\left(\lambda_{{\bf S}}^{2},\epsilon\right)>\left|\lambda_{\boldsymbol{M}}^{+}\hspace{-2pt}\left(\lambda_{{\bf S}}^{*},\epsilon\right)\right|$.
Putting it altogether, it follows that the inner optimization problem
in \eqref{eq:max_max_max} equals $\left|\lambda_{\boldsymbol{M}}^{+}\hspace{-2pt}\left(\lambda_{{\bf S}}^{2},\epsilon\right)\right|$,
for all $\epsilon>0$. Therefore, and taking into account the facts
about comparing $\left|\lambda_{\boldsymbol{M}}^{+}\hspace{-2pt}\left(\lambda_{{\bf S}},\epsilon\right)\right|$
with $\epsilon/\left(1+\epsilon\right)$, for $\lambda_{{\bf S}}\ge0$
(since we have assumed that $\lambda_{{\bf S}}^{2}\ge0$), discussed
above, we get that $\rho\left(\epsilon,\lambda_{{\bf S}}^{2},\ldots,\lambda_{{\bf S}}^{S}\right)\hspace{-2pt}\equiv\hspace{-2pt}\rho\left(\epsilon,\lambda_{{\bf S}}^{2}\right)$,
with\makeatletter 
\renewcommand*\env@cases[1][2.1]{%
	\let\@ifnextchar\new@ifnextchar   
	\left\lbrace   
	\def\arraystretch{#1}%   
	\array{@{}l@{\quad}l@{}}%
} 
\makeatother
\begin{flalign}
\hspace{-2pt}\rho\left(\epsilon,\lambda_{{\bf S}}^{2}\right)\hspace{-2pt} & \equiv\hspace{-2pt}\max\left\{ \dfrac{\epsilon}{1+\epsilon},\left|\lambda_{\boldsymbol{M}}^{+}\hspace{-2pt}\left(\lambda_{{\bf S}}^{2},\epsilon\right)\right|\right\} \nonumber \\
\hspace{-2pt} & =\hspace{-2pt}\begin{cases}
\left|\lambda_{\boldsymbol{M}}^{+}\hspace{-2pt}\left(\lambda_{{\bf S}}^{2},\epsilon\right)\right|, & \text{if }\lambda_{{\bf S}}^{2}\hspace{-2pt}\in\hspace{-2pt}{\cal R}\lor\left[\lambda_{{\bf S}}^{2}\hspace{-2pt}\in\hspace{-2pt}{\cal C},\text{ }\epsilon\hspace{-2pt}\le\hspace{-2pt}\dfrac{1+\lambda_{{\bf S}}^{2}}{1-\lambda_{{\bf S}}^{2}}\right]\\
\dfrac{\epsilon}{1+\epsilon}, & \text{if }\lambda_{{\bf S}}^{2}\hspace{-2pt}\in\hspace{-2pt}{\cal C},\text{ }\epsilon\hspace{-2pt}>\hspace{-2pt}\dfrac{1+\lambda_{{\bf S}}^{2}}{1-\lambda_{{\bf S}}^{2}}
\end{cases}\hspace{-2pt}\hspace{-2pt},\hspace{-2pt}\label{eq:F_SLEM_1}
\end{flalign}
which of course coincides with \eqref{eq:SLEM_positive}.\makeatletter 
\renewcommand*\env@cases[1][1.2]{%
	\let\@ifnextchar\new@ifnextchar   
	\left\lbrace   
	\def\arraystretch{#1}%   
	\array{@{}l@{\quad}l@{}}%
} 
\makeatother 

\textit{Suppose that} $\lambda_{{\bf S}}^{2}<0$. In this unlikely
event, except for $\lambda_{{\bf S}}^{1}$, all other eigenvalues
of ${\bf S}$ will be negative. Instead of proceeding as above, we
first observe that if $\epsilon\le1$, it will be true that $\lambda_{{\bf S}}\in{\cal R}$,
for any $\lambda_{{\bf S}}\in\left[-1,0\right)$. Second, under these
circumstances, it can be easily shown that $\left|\lambda_{\boldsymbol{M}}^{+}\hspace{-2pt}\left(\lambda_{{\bf S}},\epsilon\right)\right|\equiv\lambda_{\boldsymbol{M}}^{+}\hspace{-2pt}\left(\lambda_{{\bf S}},\epsilon\right)>\left|\lambda_{\boldsymbol{M}}^{-}\hspace{-2pt}\left(\lambda_{{\bf S}},\epsilon\right)\right|$
(due to the analysis presented above) and that $\lambda_{\boldsymbol{M}}^{+}\hspace{-2pt}\left(\lambda_{{\bf S}},\epsilon\right)$
is strictly increasing when $\epsilon<1$, whereas $\lambda_{\boldsymbol{M}}^{+}\hspace{-2pt}\left(\lambda_{{\bf S}},\epsilon\right)\equiv\epsilon/\left(\epsilon+1\right)\equiv1/2$,
if and only if $\epsilon\equiv1$. Then, again via a careful equivalence
test, it may be shown that $\left|\lambda_{\boldsymbol{M}}^{+}\hspace{-2pt}\left(\lambda_{{\bf S}},\epsilon\right)\right|\equiv\lambda_{\boldsymbol{M}}^{+}\hspace{-2pt}\left(\lambda_{{\bf S}},\epsilon\right)\ge\epsilon/\left(\epsilon+1\right)$
if $\epsilon\le1$, and the complementary if $\epsilon>1$. Thus,
what remains is to study how $\left|\lambda_{\boldsymbol{M}}^{\pm}\hspace{-2pt}\left(\lambda_{{\bf S}},\epsilon\right)\right|$
compare with $\epsilon/\left(\epsilon+1\right)$ when $\lambda_{{\bf S}}\in{\cal C}$
and $\epsilon>1$. From the generic equivalence \eqref{eq:Equivalence_1},
it follows that, simply, $\left|\lambda_{\boldsymbol{M}}^{\pm}\hspace{-2pt}\left(\lambda_{{\bf S}},\epsilon\right)\right|<\epsilon/\left(\epsilon+1\right)$,
whenever $\lambda_{{\bf S}}\in\left(-1,0\right)$. Also, when $\lambda_{{\bf S}}\equiv-1$,
$\left|\lambda_{\boldsymbol{M}}^{\pm}\hspace{-2pt}\left(\lambda_{{\bf S}},\epsilon\right)\right|\le1/\left(\epsilon+1\right)$,
which is strictly smaller than $\epsilon/\left(\epsilon+1\right)$,
since $\epsilon>1$. Combining all the above with the strict monotonicity
of $\left|\lambda_{\boldsymbol{M}}^{\pm}\hspace{-2pt}\left(\lambda_{{\bf S}},\epsilon\right)\right|$
when $\lambda_{{\bf S}}\in{\cal C}$ and working similarly to \eqref{eq:max_max_max},
we again have $\rho\left(\epsilon,\lambda_{{\bf S}}^{2},\ldots,\lambda_{{\bf S}}^{S}\right)\hspace{-2pt}\equiv\hspace{-2pt}\rho\left(\epsilon,\lambda_{{\bf S}}^{2}\right)$,
with
\begin{flalign}
\rho\left(\epsilon,\lambda_{{\bf S}}^{2}\right) & \hspace{-2pt}\equiv\hspace{-2pt}\max\left\{ \dfrac{\epsilon}{1+\epsilon},\left|\lambda_{\boldsymbol{M}}^{+}\hspace{-2pt}\left(\lambda_{{\bf S}}^{2},\epsilon\right)\right|\right\} \nonumber \\
 & \hspace{-2pt}=\hspace{-2pt}\begin{cases}
\left|\lambda_{\boldsymbol{M}}^{+}\hspace{-2pt}\left(\lambda_{{\bf S}}^{2},\epsilon\right)\right|, & \text{if }\epsilon\le1\text{ }\left(\text{and }\lambda_{{\bf S}}\in{\cal R}\right)\\
\dfrac{\epsilon}{1+\epsilon}, & \text{if }\epsilon>1
\end{cases}\hspace{-2pt},\label{eq:F_SLEM_2}
\end{flalign}
which coincides with the respective expression of Theorem \ref{Spectrum_of_M}.

We now turn our attention to the problem of minimizing the SLEM of
${\bf M}$ with respect to $\epsilon$. Given what was stated above,
we are interested in solving the minimax problem
\begin{equation}
\rho^{*}\hspace{-2pt}\left(\lambda_{{\bf S}}^{2}\right)\hspace{-2pt}\hspace{-2pt}\triangleq\hspace{-2pt}\min_{\epsilon>0}\rho\hspace{-2pt}\left(\epsilon,\lambda_{{\bf S}}^{2}\right)\hspace{-2pt}\hspace{-2pt}\equiv\hspace{-2pt}\min_{\epsilon>0}\max\hspace{-2pt}\left\{ \hspace{-2pt}\dfrac{\epsilon}{1\hspace{-2pt}+\hspace{-2pt}\epsilon},\hspace{-2pt}\left|\lambda_{\boldsymbol{M}}^{+}\hspace{-2pt}\left(\lambda_{{\bf S}}^{2},\epsilon\right)\right|\hspace{-2pt}\right\} \hspace{-2pt},\hspace{-2pt}\hspace{-2pt}
\end{equation}
for given $\lambda_{{\bf S}}^{2}\hspace{-2pt}\in\hspace{-2pt}\left(-1,1\right)$.
It is of course reasonable to consider the cases where $\lambda_{{\bf S}}^{2}\ge0$
and $\lambda_{{\bf S}}^{2}<0$ separately, as above.

\textit{Suppose that} $\lambda_{{\bf S}}^{2}\ge0$. If additionally
$\lambda_{{\bf S}}^{2}\hspace{-2pt}\in\hspace{-2pt}{\cal R}$ (see
\eqref{eq:F_SLEM_1}), $\rho\left(\epsilon,\lambda_{{\bf S}}\right)\equiv\lambda_{\boldsymbol{M}}^{+}\hspace{-2pt}\left(\lambda_{{\bf S}}^{2},\epsilon\right)$.
In the following, we show that $\rho\left(\epsilon,\lambda_{{\bf S}}^{2}\right)$
is strictly decreasing in $\epsilon$, for all $\lambda_{{\bf S}}^{2}\hspace{-2pt}\in\hspace{-2pt}{\cal R}$,
through an equivalence test (actually, \textit{either $\lambda_{{\bf S}}^{2}$
is positive or not}). Suppose that $\partial\rho\left(\epsilon,\lambda_{{\bf S}}^{2}\right)/\partial\epsilon<0$.
Then, after some algebra, this statement can be shown to be equivalent
to 
\begin{equation}
\lambda_{{\bf S}}^{2}\sqrt{1+\epsilon^{2}\left(\left(\lambda_{{\bf S}}^{2}\right)^{2}-1\right)}<1-\epsilon\left(\left(\lambda_{{\bf S}}^{2}\right)^{2}-1\right).
\end{equation}
Since both sides of the inequality above are nonnegative for $\lambda_{{\bf S}}^{2}\ge0$
(if $\lambda_{{\bf S}}^{2}<0$, then take everything on the LHS and
we are done), we can take squares and setting $x\triangleq$$\left(\lambda_{{\bf S}}^{2}\right)^{2}-1$,
we get the equivalent statement $x\left(1+\epsilon\right)^{2}<0$,
which is always true, since $x<0$, for all $\lambda_{{\bf S}}^{2}\in\left[0,1\right)$.
Thus, $\rho\left(\epsilon,\lambda_{{\bf S}}^{2}\right)$ is strictly
decreasing in $\epsilon$. Through similar procedures, we can easily
show that if $\lambda_{{\bf S}}^{2}\in{\cal C}$, then $\rho\left(\epsilon,\lambda_{{\bf S}}^{2}\right)$
is strictly increasing (for both respective subcases). Consequently,
for $\lambda_{{\bf S}}^{2}\ge0$, $\rho\left(\epsilon,\lambda_{{\bf S}}^{2}\right)$
is minimized at $\epsilon^{*}\hspace{-2pt}\left(\lambda_{{\bf S}}^{2}\right)\equiv1/\sqrt{1-\left(\lambda_{{\bf S}}^{2}\right)^{2}}$,
with optimal value $\rho^{*}\hspace{-2pt}\left(\lambda_{{\bf S}}^{2}\right)\equiv\rho\left(\epsilon^{*}\hspace{-2pt}\left(\lambda_{{\bf S}}^{2}\right),\lambda_{{\bf S}}^{2}\right)$.

\textit{Finally, suppose that} $\lambda_{{\bf S}}^{2}<0$. From the
previous discussion and \eqref{eq:F_SLEM_2}, it follows that $\rho\left(\epsilon,\lambda_{{\bf S}}^{2}\right)$
is strictly decreasing if $\epsilon\le1$ and strictly increasing
otherwise. Therefore, for $\lambda_{{\bf S}}^{2}<0$, $\epsilon^{*}\hspace{-2pt}\left(\lambda_{{\bf S}}^{2}\right)\equiv1$,
with $\rho^{*}\hspace{-2pt}\left(\lambda_{{\bf S}}^{2}\right)\equiv1/2$.
The proof is complete.\hfill{}\ensuremath{\blacksquare}

\section*{Appendix B: Proof of Theorem \ref{Optimal_Rates}}

From Theorems \ref{Lemma_ADMM_1} and \ref{Spectrum_of_M}, it is
apparent that, in order to study the behavior of the consensus error
bound with respect to the regularizer $\epsilon$ and the doubly stochastic
matrix ${\bf S}$, it suffices to look at the bivariate function
\begin{equation}
f_{n}\left(\epsilon,\lambda_{{\bf S}}^{2}\right)\triangleq\dfrac{\left(\rho\left(\epsilon,\lambda_{{\bf S}}^{2}\right)\right)^{n-1}}{1+\epsilon},\;\epsilon>0,\,\lambda_{{\bf S}}^{2}\in\left[-1,1\right),
\end{equation}
parametrized by $n\in\mathbb{N}^{2}\cap\mathbb{N}^{\left\lfloor \epsilon+1\right\rfloor }$.
Whenever it exists, the first order derivative of $f_{n}\left(\epsilon,\lambda_{{\bf S}}^{2}\right)$
with respect to $\epsilon$ is given by
\begin{equation}
\dfrac{\partial f_{n}\left(\epsilon,\lambda_{{\bf S}}^{2}\right)}{\partial\epsilon}\equiv f_{n}\left(\epsilon,\lambda_{{\bf S}}^{2}\right)\hspace{-2pt}\left(\dfrac{\left(n-1\right)}{\rho\left(\epsilon,\lambda_{{\bf S}}^{2}\right)}\dfrac{\partial\rho\left(\epsilon,\lambda_{{\bf S}}^{2}\right)}{\partial\epsilon}-\dfrac{1}{\left(1+\epsilon\right)}\right)\hspace{-2pt}.\label{eq:deriv_e}
\end{equation}
From Theorem \ref{Spectrum_of_M}, as well as its proof presented
in Appendix A above, we know that (recall the definitions of sets
${\cal R}$ and ${\cal C}$) either (a) $\lambda_{{\bf S}}^{2}\hspace{-2pt}\in\hspace{-2pt}{\cal R}$
and $\rho\left(\epsilon,\lambda_{{\bf S}}^{2}\right)$ is strictly
decreasing in $\epsilon$, or (b) $\lambda_{{\bf S}}^{2}\hspace{-2pt}\in\hspace{-2pt}{\cal C}$
and $\rho\left(\epsilon,\lambda_{{\bf S}}^{2}\right)\equiv\left|\lambda_{\boldsymbol{M}}^{+}\hspace{-2pt}\left(\lambda_{{\bf S}}^{2},\epsilon\right)\right|$
is strictly increasing in $\epsilon$, or (c) $\lambda_{{\bf S}}^{2}\hspace{-2pt}\in\hspace{-2pt}{\cal C}$
and $\rho\left(\epsilon,\lambda_{{\bf S}}^{2}\right)\equiv\epsilon/\left(1+\epsilon\right)$,
also strictly increasing in $\epsilon$. For (a), things are trivial,
since $\partial\rho\left(\epsilon,\lambda_{{\bf S}}^{2}\right)/\partial\epsilon<0$
and, thus, $\partial f_{n}\left(\epsilon,\lambda_{{\bf S}}^{2}\right)/\partial\epsilon<0$
as well, implying that $f_{n}\left(\epsilon,\lambda_{{\bf S}}^{2}\right)$
is strictly decreasing in $\epsilon$, for all $n\in\mathbb{N}^{2}\cap\mathbb{N}^{\left\lfloor \epsilon+1\right\rfloor }$.
Now, for both (b) and (c), since $\partial\rho\left(\epsilon,\lambda_{{\bf S}}^{2}\right)/\partial\epsilon>0$,
\eqref{eq:deriv_e} implies that
\begin{equation}
\dfrac{\partial f_{n}\left(\epsilon,\lambda_{{\bf S}}^{2}\right)}{\partial\epsilon}>0\Leftrightarrow n>\dfrac{\rho\left(\epsilon,\lambda_{{\bf S}}^{2}\right)}{\left(1+\epsilon\right)\partial\rho\left(\epsilon,\lambda_{{\bf S}}^{2}\right)/\partial\epsilon}+1,
\end{equation}
from where it can be carefully shown that the inequality on the right
becomes $n>2\epsilon+1$ for case (b) and $n>\epsilon+1$ for case
(c). Therefore, choosing $n>2\epsilon+1$, for any $\epsilon$ results
in a convenient, global constraint for $n$. Next, fix ${\bf S}\in\mathfrak{S}$
and choose an $\epsilon_{max}\ge\epsilon^{*}\hspace{-2pt}\left(\lambda_{{\bf S}}^{2}\right)$.
Then, as long as $n>2\epsilon_{max}+1$, it may be readily shown that
$f_{n}\left(\epsilon,\lambda_{{\bf S}}^{2}\right)$ behaves exactly
like $\rho\left(\epsilon,\lambda_{{\bf S}}^{2}\right)$ whenever $\epsilon\in\left(0,\left(n-1\right)/2\right]$,
as far as where its extrema are located, as well as their type. In
particular, since $\epsilon^{*}\hspace{-2pt}\left(\lambda_{{\bf S}}^{2}\right)\le\epsilon_{max}<\left(n-1\right)/2$
by construction, it is true that $f_{n}\left(\epsilon,\lambda_{{\bf S}}^{2}\right)$
is minimized exactly at $\epsilon^{*}\hspace{-2pt}\left(\lambda_{{\bf S}}^{2}\right)$
(as defined in Theorem \ref{Spectrum_of_M}), whenever $\epsilon\in\left(0,\epsilon_{max}\right]$,
\textit{regardless of} $n$. Observe that demanding $n>2\epsilon_{max}+1$
implies that the constraint $n\in\mathbb{N}^{2}\cap\mathbb{N}^{\left\lfloor \epsilon+1\right\rfloor }$
is satisfied, for all $\epsilon\in\left(0,\epsilon_{max}\right]$.

The case where $\epsilon>\epsilon_{max}$ is quite more complicated.
To this end, let us consider the \textit{multiobjective, }scalar,
constrained optimization problem (recall the constraint $n\in\mathbb{N}^{2}\cap\mathbb{N}^{\left\lfloor \epsilon+1\right\rfloor }$
and that $n>2\epsilon_{max}+1>2$)
\begin{equation}
\underset{\epsilon>0}{\mathrm{minimize}}\hspace{7.3pt}\left[\left\{ f_{n}\left(\epsilon,\lambda_{{\bf S}}^{2}\right)\text{ s.t. }\epsilon<n\right\} _{n>2\epsilon_{max}+1}\right].\label{eq:multi_2}
\end{equation}
In order for the problem to be feasible, every individual objective
must be feasible. Therefore, in order to satisfy the constraint $\epsilon<n$,
\textit{simultaneously for all} $n>2\epsilon_{max}+1$, it must be
true that $\epsilon\in\left(0,\left\lfloor 2\epsilon_{max}+2\right\rfloor \right)$.
Now, since the case $\epsilon\in\left(0,\epsilon_{max}\right]$ has
already been covered above, suppose that $\epsilon\in\left(\epsilon_{max},\left\lfloor 2\epsilon_{max}+2\right\rfloor \right)$.
For such $\epsilon$'s, let us compare the values of the objectives
of \eqref{eq:multi_2} with those obtained for $\epsilon\in\left(0,\epsilon_{max}\right]$,
and, in particular, for $\epsilon\equiv\epsilon^{*}\hspace{-2pt}\left(\lambda_{{\bf S}}^{2}\right)$.
For any $n>2\epsilon_{max}+1$ and whenever $\left(n-1\right)/2<\left\lfloor 2\epsilon_{max}+2\right\rfloor $,
there are two possible choices for $\epsilon$: either $\epsilon\in\left(\epsilon_{max},\left(n-1\right)/2\right]$,
where, due to monotonicity, $f_{n}\left(\epsilon,\lambda_{{\bf S}}^{2}\right)>f_{n}\left(\epsilon^{*}\hspace{-2pt}\left(\lambda_{{\bf S}}^{2}\right),\lambda_{{\bf S}}^{2}\right)$,
or $\epsilon\in\left(\left(n-1\right)/2,\left\lfloor 2\epsilon_{max}+2\right\rfloor \right)$,
which might indeed be such that $f_{n}\left(\epsilon,\lambda_{{\bf S}}^{2}\right)<f_{n}\left(\epsilon^{*}\hspace{-2pt}\left(\lambda_{{\bf S}}^{2}\right),\lambda_{{\bf S}}^{2}\right)$,
since, in this region and for the particular $n$, $f_{n}\left(\epsilon,\lambda_{{\bf S}}^{2}\right)$
might be decreasing. However, for the same $\epsilon$, there always
exists $\widetilde{n}>2\epsilon_{max}+1$ sufficiently large, corresponding
to another objective of \eqref{eq:multi_2}, such that $\epsilon\in\left(\epsilon_{max},\left(\widetilde{n}-1\right)/2\right]$,
implying, as above, that $f_{\widetilde{n}}\left(\epsilon,\lambda_{{\bf S}}^{2}\right)>f_{\widetilde{n}}\left(\epsilon^{*}\hspace{-2pt}\left(\lambda_{{\bf S}}^{2}\right),\lambda_{{\bf S}}^{2}\right)$.
In the remaining case where $\left(n-1\right)/2\ge\left\lfloor 2\epsilon_{max}+2\right\rfloor $,
again from monotonicity as above, it readily follows that $f_{n}\left(\epsilon,\lambda_{{\bf S}}^{2}\right)>f_{n}\left(\epsilon^{*}\hspace{-2pt}\left(\lambda_{{\bf S}}^{2}\right),\lambda_{{\bf S}}^{2}\right)$.
We have shown that if, for some $n>2\epsilon_{max}+1$, there exists
$\epsilon>0$ such that $f_{n}\left(\epsilon,\lambda_{{\bf S}}^{2}\right)<f_{n}\left(\epsilon^{*}\hspace{-2pt}\left(\lambda_{{\bf S}}^{2}\right),\lambda_{{\bf S}}^{2}\right)$,
then there exists another $\widetilde{n}>2\epsilon_{max}+1$, such
that, for the same $\epsilon$, the corresponding objective is relatively
increased, that is, $f_{\widetilde{n}}\left(\epsilon,\lambda_{{\bf S}}^{2}\right)>f_{\widetilde{n}}\left(\epsilon^{*}\hspace{-2pt}\left(\lambda_{{\bf S}}^{2}\right),\lambda_{{\bf S}}^{2}\right)$.
Therefore, the point $\epsilon\equiv\epsilon^{*}\hspace{-2pt}\left(\lambda_{{\bf S}}^{2}\right)$
constitutes a Pareto optimal solution for \eqref{eq:multi_2} and
the resulting optimal bound \eqref{eq:Bound_2} of Theorem \ref{Optimal_Rates}
readily follows.

Regarding the second part of Theorem \ref{Optimal_Rates}, for a fixed,
``worst case'' $\widetilde{{\bf S}}\in\mathfrak{S}$, we focus on
the function(al) (see \eqref{eq:Bound_2})
\begin{equation}
g_{n}\left(\lambda_{{\bf S}}^{2}\right)\triangleq\gamma\left(\lambda_{{\bf S}}^{2}\right)\left(\rho^{*}\left(\lambda_{{\bf S}}^{2}\right)\right)^{n},\quad\lambda_{{\bf S}}^{2}\in\left[-1,1\right),
\end{equation}
for $n>2\epsilon_{max}+1\ge2\epsilon^{*}\hspace{-2pt}\left(\lambda_{\widetilde{{\bf S}}}^{2}\right)+1$.
Via a simple but tedious equivalence test, it can be shown that
\begin{equation}
\dfrac{\partial g_{n}\left(\lambda_{{\bf S}}^{2}\right)}{\partial\lambda_{{\bf S}}^{2}}>0\Leftrightarrow n>\dfrac{\lambda_{{\bf S}}^{2}+1}{\sqrt{1-\left(\lambda_{{\bf S}}^{2}\right)^{2}}},
\end{equation}
for all $\lambda_{{\bf S}}^{2}\in\left[0,1\right)$ (note that for
$\lambda_{{\bf S}}^{2}\in\left[-1,0\right)$, $g_{n}\left(\lambda_{{\bf S}}^{2}\right)$
is constant), where it can be easily verified that the RHS of the
above stated inequality on $n$ (on the right) is itself strictly
increasing in $\lambda_{{\bf S}}^{2}$ and that, additionally (recall
that, by definition, $\epsilon_{max}\ge\epsilon^{*}\hspace{-2pt}\left(\lambda_{\widetilde{{\bf S}}}^{2}\right)$),
the already existent requirement that $n>2\epsilon_{max}+1$ implies
the aforementioned inequality, regardless of $\lambda_{{\bf S}}^{2}\in\left[0,\lambda_{\widetilde{{\bf S}}}^{2}\right]$.
This may be shown, again, via a, trivial this time, equivalence test.
Then, it is guaranteed that $g_{n}\left(\lambda_{{\bf S}}^{2}\right)$
will be increasing whenever ${\bf S}\in\mathfrak{S}$ is additionally
such that $\lambda_{{\bf S}}^{2}\in\left[-1,\lambda_{\widetilde{{\bf S}}}^{2}\right]$.
Our claims follow.\hfill{}\ensuremath{\blacksquare}

\section*{Appendix C: Proof of Lemma \ref{Lemma_ADMM_2}}

Simply, since
\begin{equation}
\left\Vert \boldsymbol{\vartheta}_{t}^{j}\left(n\right)-\boldsymbol{\vartheta}_{t}^{j}\left(\infty\right)\right\Vert _{2}^{2}\equiv\sum_{k\in\mathbb{N}_{S}^{+}}\left|\boldsymbol{\vartheta}_{t}^{j}\left(n,k\right)-\theta_{t}\left(\boldsymbol{x}_{j}\right)\right|^{2},
\end{equation}
under the respective hypotheses, Lemma \ref{Lemma_ADMM_1} implies
that
\begin{equation}
\left|\boldsymbol{\vartheta}_{t}^{j}\left(n,k\right)-\theta_{t}\left(\boldsymbol{x}_{j}\right)\right|^{2}\le\left(\gamma\left\Vert \boldsymbol{\theta}_{t}\left(\boldsymbol{x}_{j}\right)\right\Vert _{2}n\rho^{n}\right)^{2},
\end{equation}
for all $k\in\mathbb{N}_{S}^{+}$, which proves the result.\hfill{}\ensuremath{\blacksquare}

\section*{Appendix D: Proof of Lemma \ref{Lemma_GROWTH}}

By definition, it is true that
\begin{flalign}
\dfrac{\theta_{t}^{k}\left(\boldsymbol{x}_{j}\right)}{S} & \equiv\left({\bf \overline{y}}_{t}^{k}\left(\boldsymbol{x}_{j}\right)\right)^{\boldsymbol{T}}\left(\boldsymbol{\Sigma}_{t}^{k}\left(\boldsymbol{x}_{j}\right)\right)^{-1}{\bf \overline{y}}_{t}^{k}\left(\boldsymbol{x}_{j}\right)+\log\det\left(\boldsymbol{\Sigma}_{t}^{k}\left(\boldsymbol{x}_{j}\right)\right)\nonumber \\
 & \le\left\Vert {\bf y}_{t}^{k}-\boldsymbol{\mu}_{t}^{k}\left(\boldsymbol{x}_{j}\right)\right\Vert _{2}^{2}\lambda_{inf}+N_{k}\log\left(\lambda_{sup}\right)\nonumber \\
 & \le\left(\left\Vert {\bf y}_{t}^{k}\right\Vert _{2}+\mu_{sup}\right)^{2}\lambda_{inf}+N_{k}\log\left(\lambda_{sup}\right)\nonumber \\
 & \le\left(\left\Vert {\bf y}_{t}\right\Vert _{2}+\mu_{sup}\right)^{2}\lambda_{inf}+N\log\left(\lambda_{sup}\right).
\end{flalign}
However, from Lemma \ref{Lemma_WHP}, $\left\Vert {\bf y}_{t}\right\Vert _{2}\hspace{-2pt}<\hspace{-2pt}\sqrt{\beta CN\hspace{-2pt}\left(1\hspace{-2pt}+\hspace{-2pt}\log\hspace{-2pt}\left(T\hspace{-2pt}+\hspace{-2pt}1\right)\right)}$
simultaneously for all $t\in\mathbb{N}_{T}$, with probability at
least $1-\left(T\hspace{-2pt}+\hspace{-2pt}1\right)^{1-CN}\hspace{-2pt}\exp\hspace{-2pt}\left(-CN\right)$.
Consequently, in that event,
\begin{align}
\dfrac{\theta_{t}^{k}\left(\boldsymbol{x}_{j}\right)}{S} & \le\left(\sqrt{\beta CN\left(1+\log\left(T+1\right)\right)}+\mu_{sup}\right)^{2}\lambda_{inf}+N\log\left(\lambda_{sup}\right)
\end{align}
and since the bound $\sqrt{\beta CN\left(1+\log\left(t+1\right)\right)}$
is greater than unity, it will be true that
\begin{align}
\dfrac{\theta_{t}^{k}\left(\boldsymbol{x}_{j}\right)}{S} & \le\beta CN\left(1+\log\left(T+1\right)\right)\left(\left(1+\mu_{sup}\right)^{2}\lambda_{inf}+\log\left(\lambda_{sup}\right)\right).
\end{align}
Further, defining $\delta\hspace{-2pt}\triangleq\hspace{-2pt}\beta\left(\hspace{-2pt}\left(1+\mu_{sup}\right)^{2}\hspace{-2pt}\lambda_{inf}\hspace{-2pt}+\hspace{-2pt}\log\hspace{-2pt}\left(\lambda_{sup}\right)\hspace{-2pt}\right)\hspace{-2pt}>\hspace{-2pt}1$,
the $\ell_{2}$-norm of $\boldsymbol{\theta}_{t}\left(\boldsymbol{x}_{j}\right)$
can be upper bounded as 
\begin{flalign}
\left\Vert \boldsymbol{\theta}_{t}\left(\boldsymbol{x}_{j}\right)\right\Vert _{2} & \equiv\sqrt{\sum_{k\in{\cal C}}\left(\theta_{t}^{k}\left(\boldsymbol{x}_{j}\right)\right)^{2}}\nonumber \\
 & \le\delta S^{1.5}CN\left(1+\log\left(T+1\right)\right),
\end{flalign}
with probability at least $1\hspace{-2pt}-\hspace{-2pt}\left(T\hspace{-2pt}+\hspace{-2pt}1\right)^{1-CN}\hspace{-2pt}\exp\hspace{-2pt}\left(-CN\right)$.\hfill{}\ensuremath{\blacksquare}

\section*{Appendix E: Proof of Lemma \ref{Lemma_ADMM_3}}

From Lemma \ref{Lemma_ADMM_2}, we know that
\begin{equation}
\boldsymbol{\vartheta}_{t}^{j}\left(n,k\right)\ge\theta_{t}\left(\boldsymbol{x}_{L}^{j}\right)-\gamma\left\Vert \boldsymbol{\theta}_{t}\left(\boldsymbol{x}_{j}\right)\right\Vert _{2}n\rho^{n},
\end{equation}
from where, invoking Lemma \ref{Lemma_GROWTH} and by definition of
$\theta_{t}\left(\boldsymbol{x}_{j}\right)$, we arrive at the lower
bound
\begin{align}
\boldsymbol{\vartheta}_{t}^{j}\hspace{-2pt}\left(n,\hspace{-2pt}k\right) & \hspace{-2pt}\ge\hspace{-2pt}N\hspace{-2pt}\log\hspace{-2pt}\left(\lambda_{inf}\right)\hspace{-2pt}-\hspace{-2pt}\gamma\delta S^{1.5}CN\hspace{-2pt}\left(1\hspace{-2pt}+\hspace{-2pt}\log\hspace{-2pt}\left(T\hspace{-2pt}+\hspace{-2pt}1\right)\right)\hspace{-2pt}n\rho^{n},
\end{align}
being true with measure at least $1-\left(T+1\right)^{1-CN}\exp\left(-CN\right)$,
under the respective hypotheses. Now, in order to bound $\boldsymbol{\vartheta}_{t}^{j}\left(n,k\right)$
from below, it suffices to demand that
\begin{align}
N\log\left(\lambda_{inf}\right)-\gamma\delta S^{1.5}CN\left(1+\log\left(T+1\right)\right)n\rho^{n} & \ge\dfrac{N}{2}\log\left(\lambda_{inf}\right).
\end{align}
Rearranging terms and taking logarithms, the above inequality will
be true with probability at least $1-\left(T+1\right)^{1-CN}\exp\left(-CN\right)$,
provided that \eqref{eq:Lower_Bound_1} holds, which constitutes needed
to be shown.\hfill{}\ensuremath{\blacksquare}

\bibliographystyle{ieeetr}
\bibliography{IEEEabrv}

\end{document}